\documentclass[twoside,english,headsepline]{scrartcl}
\usepackage[T1]{fontenc}
\usepackage[utf8]{inputenc}
\usepackage[a4paper]{geometry}
\usepackage{amsmath,amsthm,amssymb}
\usepackage{lmodern}
\usepackage{stix}
\usepackage[kerning=true]{microtype}
\usepackage[cmyk]{xcolor}
\usepackage{babel}
\usepackage[numbers]{natbib}
\usepackage{enumitem}
\setlist[enumerate,2]{label={(\alph*)}}
\usepackage{graphicx}
\usepackage{tikz}
\usetikzlibrary{3d,patterns}
\usepackage[unicode=true,pdfusetitle,bookmarks=true,bookmarksnumbered=false,bookmarksopen=false,breaklinks=false,pdfborder={0 0 0},pdfborderstyle={},pagebackref,colorlinks,urlcolor=blue]{hyperref}
\usepackage[capitalize,noabbrev]{cleveref}

\makeatletter

\usepackage{scrlayer-scrpage}
\ohead{\pagemark}
\rehead{B. Di Martino, C. El Hassanieh, E. Godlewski, J. Guillod, \& J. Sainte-Marie}
\lohead{Hyperbolicity of a semi-Lagrangian formulation of the hydrostatic free-surface Euler system}
\ofoot{}
\usepackage{doi}

\renewcommand*{\backrefalt}[4]{%
  \ifcase #1 %
  [no citations]%
  \or
  [see p.~#2]%
  \else
  [see pp.~#2]%
  \fi
}

\newtheorem{theorem}{Theorem}[section]
\newtheorem{proposition}[theorem]{Proposition}
\newtheorem{remark}[theorem]{Remark}

\newtheorem{corollary}[theorem]{Corollary}
\newtheorem{definition}[theorem]{Definition}
\numberwithin{equation}{section}

\newcommand\R{\mathbb{R}}
\newcommand\bx{\boldsymbol{x}}
\newcommand\by{\boldsymbol{y}}
\newcommand\bu{\boldsymbol{u}}
\newcommand\bphi{\boldsymbol{\Phi}}
\newcommand\bvphi{\boldsymbol{\varphi}}
\newcommand\bomega{\boldsymbol{\omega}}
\newcommand{\1}{\mathbb{1}}
\newcommand{\subeq}[1]{\textsubscript{\textnormal{#1}}}

\DeclareMathOperator{\diag}{diag}
\DeclareMathOperator{\meas}{meas}
\DeclareMathOperator{\card}{card}
\DeclareMathOperator{\dist}{dist}

\makeatother

\begin{document}

\title{Hyperbolicity of a semi-Lagrangian formulation of the hydrostatic free-surface Euler system}

\date{December 2024}

\hypersetup{pdfauthor={Bernard Di Martino, Chourouk El Hassanieh, Edwige Godlewski, Julien Guillod, \& Jacques Sainte-Marie}}

\author{Bernard Di Martino\textsuperscript{1,2} \and Chourouk El Hassanieh\textsuperscript{1,3} \and Edwige Godlewski\textsuperscript{1,3} \and Julien Guillod\textsuperscript{1,3,4,*} \and Jacques Sainte-Marie\textsuperscript{1,3}}

\maketitle

\footnotetext[1]{Inria, Team ANGE, Inria Paris, 48 rue Barrault, 75013 Paris, France}
\footnotetext[2]{Université de Corse, CNRS, Laboratoire Sciences Pour l'Environnement (SPE), Campus Grimaldi, 20250 Corte, France}
\footnotetext[3]{Sorbonne Université, CNRS, Université Paris Cité, Laboratoire Jacques-Louis Lions (LJLL), 75005 Paris, France}
\footnotetext[4]{École Normale Supérieure, Université PSL, CNRS, Département de Mathématiques et applications, 75005 Paris, France}
\renewcommand{\thefootnote}{\fnsymbol{footnote}}
\footnotetext[1]{Corresponding author: \href{mailto:julien.guillod@sorbonne-universite.fr}{julien.guillod@sorbonne-universite.fr}}

\begin{abstract}
    By a semi-Lagrangian change of coordinates, the hydrostatic Euler equations describing free-surface sheared flows is rewritten as a system of quasilinear equations, where stability conditions can be determined by the analysis of its hyperbolic structure. The system one obtains can be written as a quasi linear system in time and horizontal variables and involves no more vertical derivatives. However, the coefficients in front of the horizontal derivatives include an integral operator acting on the new vertical variable. The spectrum of these operators is studied in detail, in particular it includes a continuous part. Riemann invariants are then determined as conserved quantities along the characteristic curves. Examples of solutions are provided, in particular stationary solutions and solutions blowing-up in finite time. Eventually, we propose an exact multilayer $\mathbb{P}_0$-discretization, which could be used to solve numerically this semi-Lagrangian system, and analyze the eigenvalues of the corresponding discretized operator to investigate the hyperbolic nature of the approximated system.
\end{abstract}
\textsf{\textbf{MSC classes}}\quad{35Q31, 35F50, 35L02, 45K05, 35P05, 35P15, 76B99}\\
\textsf{\textbf{Reference}}\quad{}Nonlinearity \textbf{38}, 015018, 2025\\
\textsf{\textbf{DOI}}\quad{}\doi{10.1088/1361-6544/ad9b65}
\newpage

\tableofcontents 

\newpage

\section{Introduction}

The classical shallow-water equations are widely used to describe irrotational flows of incompressible fluids in large-scale domains and for which the pressure is assumed to be hydrostatic \cite{gerbeau,marche,saint-venant}.
The range of applications of the shallow-water models often includes flows in coastal regions, rivers, and channels as well as atmospheric flows or debris flows~\cite{Bouchut2003,saleri,levermore,zeitlin_book}.
However the shallow-water equations, due to their depth-averaged structure, describe only the horizontal profile of the velocity and information on the vertical component of the velocity is lost.
Through the averaging process, the modeling of a $(d+1)$-dimensional flow (where $d$ represents the number of horizontal coordinates) is reduced to a set of equations in dimension $d$.
For a better approximation of some truly $(d+1)$-dimensional physical phenomena such as shear flows, the vertical contribution of the horizontal velocity must be taken into consideration.
\medskip

In this paper we shall consider the motion of an incompressible ideal fluid described by the Euler system under the assumption that the pressure is hydrostatic. This model was first introduced by Benney \cite{Benney-SomePropertiesLong1973}.
This system provides full access to the vertical velocity profile  however it is not hyperbolic, unlike the shallow-water equations.
Indeed while in the irrotational setting  the shallow-water flow  can be described by a system of hyperbolic partial differential equations, the rotational flow described by the hydrostatic Euler equations do not seem to fall under any classical classification.
However, the hydrostatic Euler system with free surface evolution that we consider can be rewritten using a specific change of coordinates referred to as the semi-Lagrangian formulation.
This formulation was first introduced by Zakharov \cite{zakharov1980} to derive an infinite system of conservation laws for shear flows in shallow water.
The transformation maps the free-surface flow onto a fixed domain with flat boundaries. Note that it differs from the classical so-called sigma-transformation~\cite{ale_telemac} and it also differs from the Lagrangian description of the flow.
The advantage of the change of variables is that it allows us to rewrite the hydrostatic Euler system in a generalized quasi-linear form.
More precisely, the rewriting gives a quasi-linear system of equations where the differential operator involves only time and horizontal derivatives but the coefficient in front of the horizontal derivatives has an integral operator in the new vertical variable.
Therefore the classical hyperbolicity condition on a matrix is replaced by an analysis of the spectrum of an operator including both multiplicative and integral terms. Unfortunately, in general, explicit expressions of the eigenvalues are difficult to obtain, and the attempt to diagonalize the hydrostatic Euler system is only partially successful.

The approach we follow was presented in \cite{teshukov1985} by Teshukov who introduced a generalized notion of hyperbolicity and extended the theory of characteristics.
With this generalized notion of hyperbolicity, sufficient criteria for the stability of flows were established in the case of a monotonic velocity profile \cite{teshukov1994} or for non-monotonic profiles satisfying further assumptions \cite{teshukov1995}. It is shown in \cite{chesnokov2017} that all shear flows with a monotone and convex velocity profile are stable, and moreover sufficient conditions for the stability of piecewise linear continuous and discontinuous velocity profiles are determined.
\medskip

Local well-posedness of the hydrostatic Euler system with free surface seems to be open in general, however results without free-surface or under some monotonicity or convexity conditions are available.
In \cite{brenier1999}, Brenier shows that the hydrostatic Euler equations without free-surface are locally solvable for smooth solutions with a strictly convex velocity profile (local Rayleigh condition) using a similar semi-Lagrangian reformulation. 
We note that the assumptions of constant slopes at the boundaries of the domain required in \cite{brenier1999} was removed later on in \cite{Masmoudi2012}, without using a semi-Lagrangian change of variable.
However, the local convexity assumption on the velocity field is required: \cite{Renardy2009} shows that the linearization around some velocity profile is ill-posed.
For analytical initial data, local existence and uniqueness was obtained in \cite{Kukavica2011} using a Cauchy–Kovalevskaya type argument.
In \cite{teshukov1992}, Teshukov studied the local well-posedness of the free-surface problem under some monotonicity and generalized hyperbolicity conditions, but the proof lacks some precision.
\medskip

It is very interesting to note that the quasilinear equations obtained by the semi-Lagrangian formulation are very similar to the equations describing particular solutions of some kinetic equation. This link between fluid equations and kinetic equations seems to originate with the work of Brenier \cite{Brenier-1980}.
This kinetic equation has different name in the literature: Vlasov equation with quasi-neutral equation \cite{Besse2011}, Vlasov-Dirac-Benney equation \cite{Bardos-VDB2013} or quasineutral limit of Vlasov-Poisson \cite{HanKwan-Quasineutral2016}.
Here we adopt the terminology Vlasov-Dirac-Benney equation, and in short it is the classical Vlasov-Poisson equation, where the Coulomb interacting potential is replaced by the Dirac measure $\delta$.
As explained in \cite[§2]{Besse2011} and \cite[§6]{Bardos-VDB2013}, a one-bump continuous profile Ansatz of the distribution function in the Vlasov-Dirac-Benney equation leads to equations, which are almost identical to the semi-Lagrangian formulation of the Euler hydrostatic equations. The only difference is that the equation coming from Vlasov-Dirac-Benney, see \cite[equation (33)]{Bardos-VDB2013} has an additional term $\rho\partial_{x}\rho$, whereas the semi-Lagrangian formulation of Euler hydrostatic, see \eqref{eq:new_system}, has not the corresponding term $H\partial_{x}H$.
Moreover, under some hyperbolicity assumptions, Besse was able in~\cite{Besse2011} to diagonalize the corresponding quasilinear equations using the Riemann invariants, more precisely he was able to prove existence and uniqueness, and that the diagonalized system is equivalent to the original one.
Even though the characteristic equations for the eigenvalues are not exactly the same in the two works (since the equations are not exactly the same), it appears that a strict convexity assumption on the velocity profile is sufficient to ensure hyperbolicity in both cases.
It would be interesting to see if the very technical argument developed by Besse to diagonalize his system, could also be applied to our slightly different quasilinear equations.
We note that well-posedness results are also proved in \cite{Bardos-VDB2013} and \cite{HanKwan-Quasineutral2016}, however using different techniques.
\medskip

The paper presents several improvements and a few novelties.
First we provide rigorous results about the existence of the semi-Lagrangian change of variable  and thus provide that the two formulations are equivalent, at least locally in time.
Secondly, we extend the results given by Teshukov in \cite{teshukov1985} on the spectrum of the given integral and multiplicative operator, by rigorously analyzing its spectrum and determining its different parts (point and discrete spectrum as well as continuous, residual and essential spectrum). We localize the spectrum more finely and determine limiting cases for which the existence of the two real eigenvalues given in \cite{chesnokov2017} is guaranteed or not.
Eventually a multilayer formulation of the transformed system is proposed and analyzed by localizing the eigenvalues of the resulting discrete system and providing conditions for having real eigenvalues.

It is interesting to note that the multilayer formulation of the transformed system is exactly the multilayer formulation without exchange term introduced by Audusse in~\cite{audusse2005}. In particular compared to \cite{audusse2005}, our derivation is very different in that no approximations are involved, except on the initial data. We also note that a piecewise linear discretization was introduced in~\cite{chesnokov2017}.
\medskip

The assumptions made to establish that the spectrum is real (hence some sort of generalized hyperbolicity) together with some numerical experiments suggest that complex eigenvalues may exist for some physical velocity profiles. It would be interesting to exhibit precisely and study these situations, in both original and semi-Lagrangian frameworks.

Because when the number of layers increases, our discrete results are consistent with those obtained at the continuous level, the multilayer formulation~\cite{JSM_JCP,JSM_M2AN,BDGSM,fernandeznieto} is interesting in practice for the numerical approximation of the transformed system and of the hydrostatic Euler system with free-surface. The simulations of geophysical flows take place over large domains in space and time and thus a compromise between stability and accuracy has to be found for their numerical approximation.  A discretization of the multilayer version of the transformed system has several advantages:
\begin{itemize}[nolistsep]
    \item only derivatives along the horizontal axis are present;
    \item no exchange term between the layers appears in the formulation;
    \item estimates of the eigenvalues are available, which eventually could provide stability conditions for the time discretization.
\end{itemize}
The numerical scheme proposed in~\cite{audusse2005} can be used but has to be adapted \emph{e.g.} for the numerical treatment of the boundary conditions, see also~\cite{art_3d}.
The price to pay when considering the transformed system is the requirement to reinterpolate the variables when the change of variable becomes singular and thus hardly invertible. Numerical analysis of the discretized system as well as numerical results and comparison with traditional multilayer formulations will be the subject of a future study.
\medskip

The paper is organized as follows.
In \cref{sec:semi_lagrangian_formulation}, we introduce the semi-Lagrangian formulation of the hydrostatic Euler system and prove that the change of variable is locally well-defined (\cref{thm:wp}) and that the two formulations are equivalent, at least locally in time (\cref{thm:change,thm:reciproque}).

In \cref{sec:particular_solutions}, we discuss some examples of more or less explicit solutions in both formulations. In particular, we emphasize that the change of variable can become ill-defined in finite-time and we underline the degree of freedom in the definition of the change of variable (\cref{subsec:stat_sols_z}). We 
analyze the stationary solutions in both formulations (\cref{subsec:stat_sols}). The shallow water limit and an example of a flow with non-trivial vorticity are given respectively in \cref{subsec:shallow_water,subsec:flow_vorticity}.

In \cref{sec:spectrum_riemann}, the spectrum of the integral and multiplicative operator appearing in front of the horizontal derivatives is analyzed and the associated Riemann invariants are discussed. In \cref{subsec:spectrum}, we characterize rigorously the different parts of the spectrum, namely, the point, discrete, continuous, residual and essential spectrum (\cref{thm:spectrum-topo,cor:spectrum}). Moreover we localize the spectrum with precise explicit bounds (\cref{prop:spectrum_localization}). In \cref{subsec:limiting_cases} we provide conditions under which two real eigenvalues exist (\cref{prop:limit_case,prop:u_convexe}) and a counterexample showing our result in nearly optimal (\cref{rem:optimal,prop:u_convexe_counterexample}). In \cref{subsec:riemann}, the Riemann invariants associated to the two real eigenvalues and to the continuous part of the spectrum are discussed. All this section is done under the assumption that the topography is flat. However, this can be generalized to variable topography as explained in \cref{subsec:avectopo}.

Finally, \cref{sec:multi_layer} is dedicated to a multilayer approximation in the vertical variable of the system in the semi-Lagrangian formulation. Surprisingly, the multilayer discretization of this system is an exact solution of the continuous system since the dependence on the vertical variable in only through an integral (\cref{cor:multi_layer_exact}). The eigenvalue of the corresponding matrix are determined (\cref{prop:discret_spectrum}) and analyzed (\cref{prop:discrete_localization,prop:discret_spectrum}). In the end, we analyze the convergence of the spectrum of the approximation towards the spectrum of the continuous case under the assumption that the velocity profile is monotonic and convex (\cref{prop:eigen_convergence}).
\medskip

Let us  note that for simplicity, most of the results concern the one-dimensional case, \emph{i.e.}, one horizontal variable ($d=1$) and one vertical variable, however some of them also concern the case of two horizontal variables ($d=2$).

\section{Semi-Lagrangian formulation of hydrostatic Euler system}\label{sec:semi_lagrangian_formulation}

The aim of this section is to rewrite the hydrostatic Euler
system in a generalized quasi-linear form using a semi-Lagrangian
formulation, and prove that both formulations are in some sense equivalent at least for short time.

\subsection{Free-surface hydrostatic Euler system}

We denote respectively by $\bx\in\R^{d}$ for $d=1,2$ and $z$ the
coordinates in the horizontal and vertical directions (see \cref{fig:intro_setup}).
The surface elevation is described by a function
$\eta(t,\bx)$ and the bottom is given by a function $z_{b}(\bx)$
independent of time; we denote by $h(t,\bx)=\eta(t,\bx)-z_{b}(\bx)$
the water depth and by $\Omega_{t}$ the domain occupied by the fluid
at time $t$: 
\begin{equation}\label{eq:Omega_t}
    \Omega_{t}=\bigl\{(\bx,z)\in\R^{d}\times\R\,:\,z_{b}(\bx)<z<\eta(t,\bx)\bigr\}.
\end{equation}
Note that the variable $\bx$ is assumed to lie in an infinite domain,
which is counter intuitive at first glance. However, this choice circumvents
the difficulties arising if boundary conditions were to be imposed
on the horizontal directions which is not the interest of this paper.
The velocity field is given by $\left(\bu(t,\bx,z),w(t,\bx,z)\right)$
where $\bu(t,\bx,z)$ and $w(t,\bx,z)$ represent the horizontal and
vertical components, respectively. 
\begin{figure}
\centering
\begin{tikzpicture}[scale=1.2]
    \draw[thick,->] (0,0,0) -- (9,0,0) node[right] {$x$};
    \draw[thick,->] (0,0,0) -- (0,4,0) node[above] {$z$};
    \draw[thick,->] (0,0,0) -- (0,0,2) node[above] {$y$};
    \begin{scope}[canvas is xy plane at z=0]
        \draw[nearly opaque, fill=gray] plot[smooth, tension=1] coordinates {(0,3) (3,3.5) (6,2.8) (8.6,3.2)} -- (8.6,0) -- (8.6,0) -- (0,0) -- cycle;
        \draw[pattern=crosshatch] plot[smooth, tension=1] coordinates { (0,0.7) (2,1.5) (5,0.5) (7,1) (8.6,0.6)} -- (8.6,0) -- (0,0) -- cycle;
    \end{scope}
    \draw[nearly opaque, fill=gray] (0,0,0) -- (0,0,1.5) -- (0,3,1.5) -- (0,3,0) -- cycle;
    \draw[nearly opaque, pattern=crosshatch] (0,0,0) -- (0,0,1.5) -- (0,0.7,1.5) -- (0,0.7,0) -- cycle;
    \draw[nearly opaque, fill=gray] (8.6,0,0) -- (8.6,0,1.5) -- (8.6,3.2,1.5) -- (8.6,3.2,0) -- cycle;
    \draw[nearly opaque, pattern=crosshatch] (8.6,0,0) -- (8.6,0,1.5) -- (8.6,0.6,1.5) -- (8.6,0.6,0) -- cycle;
    \begin{scope}[canvas is xy plane at z=1.5]
        \draw[nearly opaque, fill=gray] plot [smooth, tension=1] coordinates { (0,3) (3,3.5) (6,2.8) (8.6,3.2)} -- (8.6,0) -- (8.6,0) -- (0,0) -- cycle;
        \draw[nearly opaque, pattern=crosshatch] plot [smooth, tension=1] coordinates { (0,0.7) (2,1.5) (5,0.5) (7,1) (8.6,0.6)} -- (8.6,0) -- (0,0) -- cycle;
        \fill[semitransparent,white] (1.9,0) rectangle (2.1,1.5);
        \draw (2,0.75) node[right,fill=white,semitransparent,text opacity=1,inner sep=0.8mm] {$z_b(\bx)$};
        \draw[thick,->] (2,0) -- (2,1.5);
        \draw[thick,->] (6,0) -- (6,2.8);
        \draw (5.95,1.8) node[right] {$\eta(t,\bx)$};
        \draw[thick,<->] (4.48,0.6) -- (4.5,3.18);
        \draw (4.45,1.8) node[right] {$h(t,\bx)$};
    \end{scope}
    \draw[thick,->] (1.5,1.7,0) -- (2.3,1.7,0) node[right] {$u(t,\bx,z)$};
    \draw[thick,->] (1.5,1.7,0) -- (1.5,2.5,0) node[right] {$w(t,\bx,z)$};
    \draw[thick,->] (1.5,1.7,0) -- (1.5,1.7,1) node[right] {$v(t,\bx,z)$};
    \draw (6.85,0.75,1.5) node[right,fill=white,semitransparent,text opacity=1,inner sep=0.8mm] {\;$\boldsymbol{n}_b(\bx)$};
    \draw[thick,->] (7,1,1.5) -- (6.9,0.5,1.5);
    \draw (6.85,3,1.5) node[right,fill=white,semitransparent,text opacity=1,inner sep=0.8mm] {\;$\boldsymbol{n}_s(\bx)$};
    \draw[thick,->] (7,2.8,1.5) -- (6.9,3.2,1.5);
\end{tikzpicture}
\caption{The three dimensional set-up for the hydrostatic Euler system with free-surface, where $\bx=(x,y)$ and $\bu=(u,v)$.}
\label{fig:intro_setup}
\end{figure}

The hydrostatic Euler system is given by:
\begin{align}
    \frac{\partial\bu}{\partial t}+\bu\cdot\nabla_{\bx}\bu+w\partial_{z}\bu+g\nabla_{\bx}\eta & =-\nabla_{\bx}p^{a},\label{eq:euler_3d1}\\
    \nabla_{\bx}\cdot\bu+\partial_{z}w & =0,\label{eq:euler_3d2}
\end{align}
where $p^{a}=p^{a}(t,\bx)$ is a given function corresponding to the
atmospheric pressure. 
This system has to be completed with boundary conditions. At the free
surface we have the kinematic boundary condition 
\begin{equation}\label{eq:free_surf}
    \frac{\partial\eta}{\partial t}+\bu_{s}\cdot\nabla_{\bx}\eta-w_{s}=0,
\end{equation}
where the subscript $s$ indicates the value of the considered quantity
at the free surface, for example $w_{s}(t,\bx)=w(t,\bx,\eta(t,\bx))$.
The kinematic boundary condition at the bottom consists in a classical
no-penetration condition, 
\begin{equation}\label{eq:bottom}
    \bu_{b}\cdot\nabla_{\bx}z_{b}-w_{b}=0,
\end{equation}
where the subscript $b$ indicates the value of the considered quantity
at bottom, for example $w_{b}(t,\bx)=w(t,\bx,z_{b}(\bx))$.

Together these equations define the following Cauchy problem for $\eta,\bu,w$:
\begin{equation}\label{eq:euler_3d}
    \left\{
    \begin{aligned}
        \frac{\partial\eta}{\partial t}+\bu_{s}\cdot\nabla_{\bx}\eta-w_{s} & =0 &  & \text{ in }(0,T)\times\mathbb{R}^{d},\\
        \bu_{b}\cdot\nabla_{\bx}z_{b}-w_{b} & =0 &  & \text{ in }(0,T)\times\mathbb{R}^{d},\\
        \frac{\partial\bu}{\partial t}+\bu\cdot\nabla_{\bx}\bu+w\partial_{z}\bu+g\nabla_{\bx}\eta & =-\nabla_{\bx}p^{a} &  & \text{ in }(0,T)\times\Omega_{t},\\
        \nabla_{\bx}\cdot\bu+\partial_{z}w & =0 &  & \text{ in }(0,T)\times\Omega_{t},\\
        \bu|_{t=0} & =\bu_{0} &  & \text{ in }\Omega_{0},\\
        \eta|_{t=0} & =\eta_{0} &  & \text{ in }\mathbb{R}^{d},
    \end{aligned}
    \right.
\end{equation}
where $\bu_{0}, \eta_{0}$ are given, and  $(0,T)\times\Omega_{t}$ is a sloppy notation for the space-time
domain 
\begin{align*}
    (0,T)\times\Omega_{t} & =\bigl\{(t,\bx,z)\in(0,T)\times\mathbb{R}^{d}\times\mathbb{R}\,:\,(\bx,z)\in\Omega_{t}\bigr\}\\
    & =\bigl\{(t,\bx,z)\in(0,T)\times\mathbb{R}^{d}\times\mathbb{R}\,:\,z_{b}(\bx)<z<\eta(t,\bx)\bigr\}.
\end{align*}

\begin{remark} \label{rem:w_from_divu}
The variable $w$ can be eliminated by integrating the divergence-free condition \eqref{eq:euler_3d}\subeq{4}:
\begin{equation}\label{eq:w}
    w = w_b - \int_{z_b(\bx)}^{z}(\nabla_{\bx}\cdot{\bu})\mathrm{d}Z
      = \bu_b\cdot\nabla_{\bx}z_b - \int_{z_b(\bx)}^{z}(\nabla_{\bx}\cdot{\bu})\mathrm{d}Z.
\end{equation}
In particular, we have
\begin{equation} \label{eq:ws_div}
w_{s}=w_{b}-\int_{z_{b}(\bx)}^{\eta(t,\bx)}(\nabla_{\bx}\cdot\bu)\,\mathrm{d}Z,
\end{equation}
so after some calculations one gets the following conservative form
\begin{align} 
    \frac{\partial\eta}{\partial t}+\nabla_{\bx}\cdot\left(\int_{z_{b}(\bx)}^{\eta(t,\bx)}\bu(t,\bx,Z)\,\mathrm{d}Z\right) & =0,\label{eq:free_surf_conserv}\\
    \frac{\partial\bu}{\partial t}+\nabla_{\bx}\cdot(\bu\otimes\bu)+\partial_{z}(w\bu)+g\nabla_{\bx}\eta & =-\nabla_{\bx}p^{a}.\label{eq:euler_conserv}
\end{align}
For $d=1$, equations \eqref{eq:free_surf_conserv} and \eqref{eq:euler_3d1} with $w$ given by \eqref{eq:w} are called the Benney equations \cite{Benney-SomePropertiesLong1973}.
\end{remark}

\begin{remark} \label{rem:vorticity}
For $d=1$, the hydrostatic vorticity is defined by $\omega=\partial_{z}u$ and satisfies the equation
\[
    \frac{\partial \omega}{\partial t}+u\partial_{x}\omega+w\partial_{z}\omega = 0,
\]
whereas for $d=2$ it is defined by $\bomega=\partial_{z}\bu^\perp$ and satisfies the equation
\[
   \frac{\partial \bomega}{\partial t}+\bu\cdot\nabla_{\bx}\bomega+w\partial_{z}\bomega-\bomega\cdot\nabla_{\bx}\bu-\left(\nabla\wedge\bu\right)\partial_{z}\bu=0.
\]
\end{remark}

From now on, we assume that $p^{a}$
is a constant unless otherwise stated (as in \cref{subsec:flow_vorticity}).

\subsection{Derivation of the semi-Lagrangian formulation}

The aim is now to reformulate \eqref{eq:euler_3d} using semi-Lagrangian
coordinates. Following \cite{teshukov1995,Zakharov-Benneyequations1981}, let us introduce a new
variable $\lambda\in I=(0,1)$ and the function $\phi=\phi(t,\bx,\lambda)$
solution of the following Cauchy problem in $(0,T)\times \tilde{\Omega}$ where $\tilde{\Omega}=\R^d\times I$:
\begin{equation}\label{eq:def_phi}
    \left\{ 
    \begin{aligned}
        \frac{\partial\phi}{\partial t}+\bu(t,\bx,\phi)\cdot\nabla_{\bx}\phi & =w(t,\bx,\phi) &  & \text{ in } (0,T)\times \tilde{\Omega},\\
        \phi(0,\bx,\lambda) & =\phi_{0}(\bx,\lambda) &  & \text{ in } \tilde{\Omega}.
    \end{aligned}
    \right.
\end{equation}
Integrating the divergence-free condition \eqref{eq:euler_3d}\subeq{4}
from $z_{b}(\bx)$ to $\phi(t,\bx,\lambda)$ and using \eqref{eq:euler_3d}\subeq{2},
the first equation in \eqref{eq:def_phi} can be rewritten in a conservative
form similar to \eqref{eq:free_surf_conserv} 
\begin{equation}\label{eq:phi_conser}
    \frac{\partial\phi}{\partial t}+\nabla_{\bx}\cdot\left(\int_{z_{b}(\bx)}^{\phi(t,\bx,\lambda)}\bu(t,\bx,z)\,\mathrm{d}z\right)=0.
\end{equation}
The initial condition in \eqref{eq:def_phi} is chosen to verify:
\begin{align}\label{eq:cond_phi0}
    \phi_{0}(\bx,0) & =z_{b}(\bx), &
    \phi_{0}(\bx,1) & =\eta_{0}(\bx), & \partial_\lambda\phi_{0}(\bx,\lambda) & >0.
\end{align}
We note that there is not a unique choice of $\phi_{0}$; a canonical
choice is 
\begin{equation}\label{eq:phi0_canonic}
    \phi_{0}(\bx,\lambda)=(1-\lambda)z_{b}(\bx)+\lambda\eta_{0}(\bx).
\end{equation}

Since the equation \eqref{eq:def_phi} is quasilinear, its local in time
well-posedness stated in the following theorem is rather standard and its proof will be postponed
to the end of the section. Below, we denote by $C^s_b$ the space of $C^s$ functions whose all derivatives up to order $s$ are continuous and bounded.
\begin{theorem}
\label{thm:wp}Let $s\geq1$, $T>0$, $p^{a}=0$, and $z_b\in C_{b}^{s}(\R^{d})$.
If $(\eta,\bu,w)$ is a solution
of \eqref{eq:euler_3d} such that $\eta\in C_{b}^{s}((0,T)\times\R^{d})$,
$\bu\in C_{b}^{s}((0,T)\times\Omega_{t})^d$, and $w\in C_{b}^{s}((0,T)\times\Omega_{t})$,
then for $\phi_{0}\in C_{b}^{s}(\tilde{\Omega})$, there exists $T^{*}\in(0,T]$
such that \eqref{eq:def_phi} admits a unique solution $\phi\in C^{s}((0,T^{*})\times\tilde{\Omega})$.

Moreover if $\inf_{(\bx,\lambda)\in\R^{d}\times I}\partial_{\lambda}\phi_{0}(\bx,\lambda)>0$
then for $t\in(0,T^{*})$ we have $\partial_{\lambda}\phi(t,\bx,\lambda)>0$,
so in particular
\begin{align*}
    (\mathrm{Id},\phi(t)):\tilde{\Omega} & \to\Omega_{t}\\
    (\bx,\lambda) & \mapsto(\bx,\phi(t,\bx,\lambda))
\end{align*}
is an orientation preserving $C^{s}$-diffeomorphism, provided it
is a diffeomorphism at $t=0$, \emph{i.e.} provided that $\phi_{0}(\bx,0)=z_{b}(\bx)$
and $\phi_{0}(\bx,1)=\eta_{0}(\bx)$.
\end{theorem}
Using this result, we can perform a semi-Lagrangian change
of variable and obtain the following reformulation of the hydrostatic
Euler system.
\begin{theorem}
\label{thm:change}Assuming that the hypotheses of the previous theorem
are satisfied with $s\geq2$, then the functions defined by
\begin{align}
    H(t,\bx,\lambda) & =\partial_{\lambda}\phi(t,\bx,\lambda), & \tilde{\bu}(t,\bx,\lambda) & =\bu(t,\bx,\phi(t,\bx,\lambda))\label{eq:def_H_utilde}
\end{align}
have regularities $H\in C^{s-1}((0,T^{*})\times\tilde{\Omega})$
and $\tilde{\bu}\in C^{s}((0,T^{*})\times\tilde{\Omega})^d$ and are
solutions of
\begin{equation} \label{eq:new_system}
    \left\{ \begin{aligned}\frac{\partial H}{\partial t}+\nabla_{\bx}\cdot(H\tilde{\bu}) & =0 &  & \text{ in }(0,T)\times\tilde{\Omega},\\
    \frac{\partial\tilde{\bu}}{\partial t}+\tilde{\bu}\cdot\nabla_{\bx}\tilde{\bu}+g\nabla_{\bx}\int_{0}^{1}H\,\mathrm{d}\lambda & =-g\nabla_{\bx}z_{b} &  & \text{ in }(0,T)\times\tilde{\Omega},\\
    H(0,\bx,\lambda) & =H_{0}(\bx,\lambda) &  & \text{ in }\tilde{\Omega},\\
    \tilde{\bu}(0,\bx,\lambda) & =\tilde{\bu}_{0}(\bx,\lambda) &  & \text{ in }\tilde{\Omega},
\end{aligned}
\right.
\end{equation}
where the initial data are given by
\begin{align} \label{eq:new_system_init}
    H_{0}(\bx,\lambda) & =\partial_{\lambda}\phi_{0}{\displaystyle (\bx,\lambda),} & \tilde{\bu}_{0}(\bx,\lambda) & =\bu(0,\bx,\phi_{0}(\bx,\lambda)).
\end{align}
\end{theorem}

We can also go back from the new formulation to the original system.
\begin{theorem}\label{thm:reciproque}
Let $s\geq1$, $T>0$, and $z_{b}\in C_{b}^{s}(\R^{d})$. If $H\in C_{b}^{s}((0,T)\times\tilde{\Omega})$
and $\tilde{\bu}\in C_{b}^{s+1}((0,T)\times\tilde{\Omega})^d$ are solutions
of \eqref{eq:new_system} with $\inf_{(\bx,\lambda)\in\R^{d}\times I}H_{0}(\bx,\lambda)>0$, then 

\begin{align*}
    (\mathrm{Id},\phi(t)):\tilde{\Omega} & \to\Omega_{t}\\
    (\bx,\lambda) & \mapsto(\bx,\phi(t,\bx,\lambda))
\end{align*}
is an orientation preserving $C^{s}$-diffeomorphism for $t\in(0,T)$
where $\phi\in C_{b}^{s}((0,T)\times\tilde{\Omega})$ and $\Omega_{t}$ are defined by
\begin{align*}
    \phi(t,\bx,\lambda) & =z_{b}(\bx)+\int_{0}^{\lambda}H(t,\bx,\lambda')\,\mathrm{d}\lambda',\\
    \Omega_{t} & =\{(\bx,z)\in\R^{d}\times\R\,:\,z_{b}(\bx)<z<\phi(t,\bx,1)\}.
\end{align*}
Moreover, if $(\mathrm{Id},\phi(t)^{-1})$ denotes the inverse of $(\mathrm{Id},\phi(t))$,
then
\begin{align*}
\eta(t,\bx) & =\phi(t,\bx,1),\\
\bu(t,\bx,z) & =\tilde{\bu}(t,\bx,\phi^{-1}(t,\bx,z)),\\
w(t,\bx,z) & =\tilde{\bu}(t,\bx,0)\cdot\nabla_{\bx}z_{b}-\int_{z_{b}(\bx)}^{z}(\nabla_{\bx}\cdot\bu)\,\mathrm{d}Z,
\end{align*}
have regularities $\eta\in C_{b}^{s}((0,T)\times\R^{d})$, $\bu\in C^{s}((0,T)\times\tilde{\Omega})^d$
and $w\in C^{s}((0,T)\times\tilde{\Omega})$ and are solutions
of \eqref{eq:euler_3d} with
\begin{align*}
\eta_{0}(\bx) & =z_{b}(\bx)+\phi(0,\bx,1)\,\\
\bu_{0}(\bx,z) & =\tilde{\bu}_{0}(\bx,\phi^{-1}(0,\bx,z))
\end{align*}
\end{theorem}

\begin{remark}
We note that to go from the original system \eqref{eq:euler_3d}  to the semi-Lagrangian formulation \eqref{eq:new_system} an initial data $\phi_0$ satisfying \eqref{eq:cond_phi0} has to be chosen. In the opposite direction, going from the new formulation \eqref{eq:new_system} to the original Euler system \eqref{eq:euler_3d}, one does not have this degree of freedom.

Note also the relation between $h=\eta-z_b$ and $H$:
\[
    h(\bx,t)= \int_{0}^{1}H(t,\bx,\lambda)\,\mathrm{d}\lambda.
\]
\end{remark}

\begin{remark}
The previous results require the existence of a solution to \eqref{eq:euler_3d}, which is currently an open problem. Even without free-surface, the hydrostatic Euler equations are well-posed only for analytical data \cite{Kukavica2011,Renardy2009}, and it is conceivable to expect similar results in the case of a free surface.
\end{remark}

\begin{remark}
We note that as for the hydrostatic Euler system \eqref{eq:euler_3d}, no local energy balance holds for \eqref{eq:new_system}.
Indeed, after some lengthy but straightforward calculations, one gets
\[
    \frac{\partial}{\partial t}\left(\frac{H|\tilde{\bu}|^{2}}{2}+gz_{b}H\right)+g\int_{0}^{1}H\mathrm{d}\lambda\frac{\partial H}{\partial t}+\nabla_{\bx}\cdot\left(H\tilde{\bu}\left(\frac{|\tilde{\bu}|^{2}}{2}+gz_{b}+g\int_{0}^{1}H\mathrm{d}\lambda\right)\right)=0,
\]
which is not a local energy equality.
However by integrating this equation in $\lambda$ we get the following energy balance:
\[
    \frac{\partial}{\partial t}\left(\int_{0}^{1}H\left(\frac{|\tilde{\bu}|^{2}}{2}+gz_{b}+\frac{g}{2}\int_{0}^{1}H\mathrm{d}\lambda\right)\mathrm{d}\lambda\right)+\nabla_{\bx}\cdot\left(\int_{0}^{1}H\tilde{\bu}\left(\frac{|\tilde{\bu}|^{2}}{2}+gz_{b}+g\int_{0}^{1}H\mathrm{d}\lambda\right)\mathrm{d}\lambda\right)=0.
\]
\end{remark}

Finally, the vorticity defined in \cref{rem:vorticity} can also be expressed in the new domain together with the corresponding vorticity equation.
\begin{proposition} \label{prop:vorticity_new}
For $d=1$, the hydrostatic vorticity $\omega=\partial_z u$ of the flow in the domain $\tilde{\Omega}$ is:
\[
    \tilde{\omega}(t,\bx,\lambda) = \omega(t,\bx,\phi(t,x,\lambda)) = \frac{\partial_\lambda\tilde{u}}{H},
\]
and satisfies the following transport equation:
\[
    \frac{\partial\tilde{\omega}}{\partial t} + \tilde{u}\frac{\partial \tilde{\omega}}{\partial x} = 0.
\]
For $d=2$, the hydrostatic vorticity of the flow in the domain $\tilde{\Omega}$ is:
\[
    \tilde{\bomega}(t,\bx,\lambda) = \bomega(t,\bx,\phi(t,\bx,\lambda)) = \frac{\partial_\lambda\tilde{\bu}^{\perp}}{H},
\]
and satisfies the following equation:
\[
    \frac{\partial \tilde{\bomega}}{\partial t} + \tilde{\bu}\cdot\nabla_{\bx}\tilde{\bomega}-\tilde{\bomega}\cdot\nabla_{\bx}\tilde{\bu}-\frac{\nabla\wedge\tilde{\bu}}{H}\partial_{\lambda}\tilde{\bu}=0.
\]
\end{proposition}
\begin{proof}
For $d=1$, an explicit calculation leads to
\[
    \partial_t \tilde{\omega} + \tilde{u}\partial_x\tilde{\omega}
    = \frac{1}{H}\partial_\lambda\left(\partial_{t}\tilde{u}+\tilde{u}\partial_x\tilde{u}+g\partial_x\int_0^1H\mathrm{d}\lambda+g\partial_{x}z_{b}\right) - \frac{\tilde{\omega}}{H}\Bigl(\partial_{t}H+\partial_{x}(H\tilde{u})\Bigr),
\]
whereas for $d=2$ we have
\begin{multline*}
    \partial_{t}\tilde{\bomega}+\tilde{\bu}\cdot\nabla_{\bx}\tilde{\bomega}-\tilde{\bomega}\cdot\nabla_{\bx}\tilde{\bu}-\frac{\nabla\wedge\tilde{\bu}}{H}\partial_{\lambda}\tilde{\bu}\\
    =\frac{1}{H}\partial_{\lambda}\left(\partial_{t}\tilde{\bu}+\tilde{\bu}\cdot\nabla_{\bx}\tilde{\bu}+g\nabla_{\bx}\int_{0}^{1}H\,\mathrm{d}\lambda+g\nabla_{\bx}z_{b}\right)^{\perp}-\frac{\tilde{\bomega}}{H}\left(\partial_{t}H+\nabla\cdot(H\tilde{\bu})\right).
\end{multline*}
Therefore, using the equations of \eqref{eq:new_system}, the end of the proof is straightforward.
\end{proof}

\subsection{Proofs of the equivalence of the two formulations}

We begin by the proof of the local well-posedness for \eqref{eq:def_phi}.

\begin{proof}[Proof of \cref{thm:wp}]
The equation being quasilinear, we use the method of characteristics.
One difficulty is that $\eta(t)$, $\bu(t)$ and $w(t)$ are only
defined on $\Omega_{t}$ and not on $\mathbb{R}^{d+1}$. Hence, we
extend the fields $w(t)$ and $\bu(t)$ defined on the moving domain
$\Omega_{t}$ to $\mathbb{R}^{d+1}$. 
By using \cite[Theorem 1]{fefferman2}, it is possible to define $C^s$ extensions of $\bu$ and $w$ on $(0,T)\times\R^{d+1}$ in a way that their norms remain bounded by a constant depending only on $s$ and $d$.
Therefore from now on we assume that $\bu\in C_{b}^{s}((0,T)\times\R^{d+1})^d$,
and $w\in C_{b}^{s}((0,T)\times\R^{d+1})$.

Let us introduce the characteristics $\boldsymbol{X}(t,\by,v)$ and $\Phi(t,\by,v)$
defined respectively as the solutions of 
\begin{align*}
    \frac{d\boldsymbol{X}(t,\by,v)}{dt} & =\bu(t,\boldsymbol{X}(t,\by,v),\Phi(t,\by,v)), & \frac{d\Phi(t,\by,v)}{dt} & =w(t,\boldsymbol{X}(t,\by,v),\Phi(t,\by,v)),\\
    \boldsymbol{X}(0,\by,v) & =\by, & \Phi(0,\by,v) & =v,
\end{align*}
and we know by the Cauchy-Lipschitz theorem that $\boldsymbol{X}\in C^{s}((0,T)\times\R^{d+1})^d$ and $\Phi\in C^{s}((0,T)\times\R^{d+1})$.\\
If $\phi$ is a solution of \eqref{eq:def_phi}, then
\begin{align*}
    \phi(t,\boldsymbol{X}(t,\by,v),\lambda) & =\Phi(t,\by,v), & \text{with} &  & v & =\phi_{0}(\by,\lambda).
\end{align*}
Since $\boldsymbol{X}(0,\by,\phi_{0}(\by,\lambda))=\by$, the map $\by\mapsto\boldsymbol{X}(t,\by,\phi_{0}(\by,\lambda))$ is
invertible at least on some small time interval $(0,T^{*})$. Therefore, denoting
its inverse by $\boldsymbol{Y}(t,\bx,\lambda)$, the solution of \eqref{eq:def_phi} is given
by
\[
    \phi(t,\bx,\lambda)=\Phi(t,\boldsymbol{Y}(t,\bx,\lambda),\phi_{0}(\boldsymbol{Y}(t,\bx,\lambda),\lambda)),
\]
and by composition of $C^{s}$ functions, we get $\phi\in C^{s}((0,T^{*})\times\tilde{\Omega})$.

If $\inf_{(\bx,\lambda)\in\tilde{\Omega}}\partial_{\lambda}\phi_{0}(\bx,\lambda)>0$,
there exists some $\varepsilon>0$ such that $\partial_{\lambda}\phi_{0}(\bx,\lambda)\geq\varepsilon$ for all $\bx\in\R^{d}$ and $\lambda\in I$.
Then the regularity on $\phi$ ensures
that there exists some $T^{*}\in(0,T]$ such that for $t\in(0,T^{*})$,  $\partial_{\lambda}\phi(t,\bx,\lambda)\geq\frac{\varepsilon}{2}$.
The gradient of the function $(\mathrm{Id},\phi(t))$ is
\[
    \nabla_{\bx,\lambda}
    \begin{pmatrix}\bx\\
        \phi(t,\bx,\lambda)
    \end{pmatrix}
    =
    \begin{pmatrix}
        \boldsymbol{1} & \boldsymbol{0}\\
        \nabla_{\bx}\phi(t,\bx,\lambda) & \partial_{\lambda}\phi(t,\bx,\lambda)
    \end{pmatrix},
\]
which is invertible for $t\in(0,T^{*})$, so $(\mathrm{Id},\phi(t))$
is an orientation preserving diffeomorphism from $\tilde{\Omega}$
to $\Omega_{t}$ defined by
\[
    \Omega_{t}=\{(\bx,z)\in\R^{d}\times\R\,:\,\phi(t,\bx,0)<z<\phi(t,\bx,1)\}.
\]
It remains to show that this definition of $\Omega_{t}$ is equivalent
to the one given by \eqref{eq:Omega_t}. Since $\bu(t,\bx,z_{b}(\bx))=\bu_{b}(t,\bx)$,
and $w(t,\bx,z_{b}(\bx))=w_{b}(t,\bx)$ satisfies \eqref{eq:euler_3d}\subeq{2},
$z_{b}(\bx)$ is a solution of \eqref{eq:def_phi} for $\lambda=0$,
hence by uniqueness $\phi(t,\bx,0)=z_{b}(\bx)$. Similarly, since
$\eta(t,\bx)$ satisfies \eqref{eq:euler_3d}\subeq{1}, $\eta(t,\bx)$
is a solution of \eqref{eq:def_phi} for $\lambda=1$, hence by uniqueness
$\phi(t,\bx,1)=\eta(t,\bx)$.
\end{proof}

Then the transformation of \eqref{eq:euler_3d} to the new system \eqref{eq:new_system} can be easily deduced.

\begin{proof}[Proof of \cref{thm:change}]
The claimed regularities on $H$ and $\tilde{\bu}$ follow from the
regularity on $\phi$ and $\bu$.

We have
\begin{equation}
\begin{aligned}\partial_{\lambda}\left(\partial_{t}\phi+\tilde{\bu}\cdot\nabla_{\bx}\phi-\tilde{w}\right) & =\partial_{t}H+\tilde{\bu}\cdot\nabla_{\bx}H+H\partial_{z}\bu\cdot\nabla_{\bx}\phi-H\partial_{z}w\\
 & =\partial_{t}H+\nabla_{\bx}\cdot(H\tilde{\bu})-H(\nabla_{\bx}\cdot\bu+\partial_{z}w),
\end{aligned}
\label{eq:equiv_H}
\end{equation}
since
\[
\nabla_{\bx}\cdot(H\tilde{\bu})=\tilde{\bu}\cdot\nabla_{\bx}H+H\nabla_{\bx}\cdot\bu+H\partial_{z}\bu\cdot\nabla_{\bx}\phi.
\]
Therefore, using \eqref{eq:def_phi}\subeq{1} and the divergence-free
condition \eqref{eq:euler_3d}\subeq{4}, we deduce the first equation
\eqref{eq:new_system}\subeq{1}.

It follows from \eqref{eq:def_H_utilde} that
\begin{align*}
\partial_{t}\tilde{\bu} & =\partial_{t}\bu+\partial_{t}\phi\,\partial_{z}\bu=\partial_{t}\bu+w\partial_{z}\bu-\bu\cdot\nabla_{\bx}\phi\,\partial_{z}\bu\\
\tilde{\bu}\cdot\nabla_{\bx}\tilde{\bu} & =\bu\cdot\nabla_{\bx}\bu+\bu\cdot\nabla_{\bx}\phi\,\partial_{z}\bu
\end{align*}
and therefore
\begin{equation}
\partial_{t}\tilde{\bu}+\tilde{\bu}\cdot\nabla_{\bx}\tilde{\bu}=\partial_{t}\bu+\bu\cdot\nabla_{\bx}\bu+w\partial_{z}\bu.\label{eq:equiv_u}
\end{equation}
Finally, since $\eta(t,\bx)=\phi(t,\bx,1)$ and $z_{b}(\bx)=\phi(t,\bx,0)$
as seen in \cref{thm:wp}, we have
\[
\eta(t,\bx)=\phi(t,\bx,1)=\phi(t,\bx,0)+\int_{0}^{1}\partial_{\lambda}\phi(t,\bx,\lambda)\,\mathrm{d}\lambda=z_{b}(\bx)+\int_{0}^{1}H(t,\bx,\lambda)\,\mathrm{d}\lambda,
\]
so
\begin{equation}
\nabla_{\bx}\eta=\nabla_{\bx}z_{b}+\nabla_{\bx}\int_{0}^{1}H(t,\bx,\lambda)\,\mathrm{d}\lambda,\label{eq:equiv_eta}
\end{equation}
and the second equation \eqref{eq:new_system}\subeq{2} is deduced using
\eqref{eq:euler_3d}\subeq{3}.
\end{proof}

Finally, the proof of converse is in the same spirit.

\begin{proof}[Proof of \cref{thm:reciproque}]
Equation \eqref{eq:new_system}\subeq{1} can be viewed as a linear transport
equation for $H$ with $\tilde{\bu}$ known,
so can be solved by the method of characteristics.
From the expression of the solution, we directly obtain that for $t\in(0,T)$
\[
 \inf_{(\bx,\lambda)\in\tilde{\Omega}} H(t,\bx,\lambda) \geq \inf_{(\bx,\lambda)\in\tilde{\Omega}} H_0(\bx,\lambda) \exp \left(-t \|\nabla\cdot\tilde{\bu}\|_{L^\infty((0,T)\times\tilde{\Omega})} \right)>0.
\]

Therefore the map $(\mathrm{Id},\phi(t))$ is an orientation preserving $C^{s}$-diffeomorphism
for $t\in(0,T)$. Then, we can define the inverse $\phi(t)^{-1}$
such that $\phi(t)^{-1}(\bx,\phi(t,\bx,\lambda))=\lambda$ and $\phi(t,\bx,\phi^{-1}(t,\bx,z))=z$.
The claimed regularity of $\eta$, $\bu$ and $w$ follows by standard
arguments. It remains to prove that $\eta$, $\bu$ and $w$ are solutions
of \eqref{eq:euler_3d}. One deduces directly that equations  \eqref{eq:euler_3d}\subeq{2,4}
are satisfied. Since $H=\partial_{\lambda}\phi$, using \eqref{eq:equiv_H}
we have
\[
\partial_{\lambda}\left(\partial_{t}\phi+\tilde{\bu}\cdot\nabla_{\bx}\phi-\tilde{w}\right)=\partial_{t}H+\nabla_{\bx}\cdot(H\tilde{\bu})=0,
\]
and therefore
\[
\partial_{t}\eta+\bu_{s}\cdot\nabla_{\bx}\eta-w_{s}=\left.\partial_{t}\phi+\tilde{\bu}\cdot\nabla_{\bx}\phi-\tilde{w}\right|_{\lambda=1}=\left.\partial_{t}\phi+\tilde{\bu}\cdot\nabla_{\bx}\phi-\tilde{w}\right|_{\lambda=0}=0,
\]
which gives \eqref{eq:euler_3d}\subeq{1}.
In view of \eqref{eq:equiv_u}-\eqref{eq:equiv_eta}
we deduce \eqref{eq:euler_3d}\subeq{3} from \eqref{eq:new_system}\subeq{2}.
\end{proof}

\section{Particular solutions}\label{sec:particular_solutions}

In general, solutions of non-linear hyperbolic equations may become discontinuous after a finite critical time at which the space derivative of the transported quantity blows up, see \cite{lefloch2002}. This may occur for the transformation $\phi(t,\bx,\lambda)$ solution of \eqref{eq:def_phi}, in which case  it will eventually evolve towards a discontinuous solution, as we will see just below. Indeed, in this section, we discuss some properties of particular solutions and the link between the two formulations.

More precisely, we will consider first stationary solutions of \eqref{eq:euler_3d} depending only on the vertical variable $z$, then we characterize the stationary solutions of \eqref{eq:new_system} when $d=1$. We also consider the shallow water regime and finally exhibit an explicit solution with vorticity.
These particular solutions allow us to highlight some  properties of the two formulations and the link between them.

\subsection{Stationary solutions of hydrostatic Euler system depending only on \texorpdfstring{$z$}{z}}\label{subsec:stat_sols_z}
In this subsection, we assume $z_b$ to be constant (null for simplicity).
If $(\eta,\bu,w)$ is a stationary solution of \eqref{eq:euler_3d} such that $\bu$ does not depend on $\bx$, then there exist a constant $\eta_0>0$ and a function $\boldsymbol{f}:\mathbb{R}\to\mathbb{R}^d$ such that
\begin{align*}
    \eta(t,\bx) &= \eta_0,&
    \boldsymbol{u}(t,\bx,z) &= \boldsymbol{f}(z),&
    w(t,\bx,z) &= 0.
\end{align*}
In this case, \eqref{eq:def_phi} defining the evolution of $\phi$ reduces to
\begin{equation} \label{eq:phi_case1}
    \left\{ 
    \begin{aligned}
        \frac{\partial\phi}{\partial t}+\boldsymbol{f}(\phi)\cdot\nabla_{\bx}\phi & = 0\\
        \phi(0,\bx,\lambda) & =\phi_{0}(\bx,\lambda).
    \end{aligned}
    \right.
\end{equation}
We assume that $\boldsymbol{f}$ is regular enough.
However, several choices for the initial condition $\phi_0$ are possible as will be discussed below:
\begin{itemize}
    \item Taking the initial condition 
\[
    \phi_0(\bx,\lambda) = \varphi_{0}(\lambda),
\]
with $\varphi_{0}(0) = 0$, $\varphi_{0}(1)=\eta_0$ and $\varphi'_0>0$, the solution of \eqref{eq:phi_case1} remains constant in $t,\bx$ and is given by
\[
    \phi(t,\bx,\lambda) = \varphi_{0}(\lambda).
\]

    \item For a more general example in the one-dimensional case $d=1$, we can consider 
\[
\phi_{0}(x, \lambda) = \lambda (\eta_0+(1-\lambda)a(x)),
\]
for some arbitrary function $a\in C^1_b(\mathbb{R})$. The condition $\partial_\lambda \phi_0>0$ holds provided $\Vert a \Vert_{L^\infty} < \eta_0$. In particular taking $\boldsymbol{f}(\phi)=\phi$ leads to a Burgers equation.
The time $T$ for which a smooth solution $\phi$ exists is given by the method of characteristics, see \cite[Chapter I]{godlewski-book}.
More precisely, since
\[
    \partial_{x}\phi_{0}(x, \lambda) = \lambda(1-\lambda)a^\prime(x),
\]
one has
\[
    \frac{1}{T} = - \lambda(1-\lambda) \inf \bigl(\{0\} \cup a^\prime(\mathbb{R}) \bigr).
\]
This means for each $\lambda$ in $I$ there is a blow-up time $T<\infty$ after which $\phi$ is no longer well-defined if and only if $\inf_\R a^\prime < 0$.
\end{itemize}
This example enlightens the dependence of the change of variable $\phi$ with respect to the choice of the initial data $\phi_0$. 
Besides, this  proves that the local well-posedness established in \cref{thm:wp} cannot be global in general.
It also illustrates the fact that a stationary solution of  \eqref{eq:euler_3d} in the domain $\Omega_t$ does not necessarily correspond to a stationary solution of \eqref{eq:new_system} in the domain $\tilde\Omega$.

\subsection{Stationary solutions of the semi-Lagrangian formulation} \label{subsec:stat_sols}
It is important to note that the original Euler system \eqref{eq:euler_3d} and the semi-Lagrangien formulation \eqref{eq:new_system} do not share the same stationary solutions.
While stationary solutions of \eqref{eq:new_system} are stationary solutions of \eqref{eq:euler_3d}, the converse is not true in general, as enlightened in \cref{subsec:stat_sols_z}.
As seen before, if $(\eta,\bu,w)$ is a stationary solution of \eqref{eq:euler_3d},
then $\phi$ might depend on time according to the choice of the initial data $\phi_{0}$. However, if the equation
\[
    \left\{ \begin{aligned}\bu(\bx,\phi(\bx,\lambda))\cdot\nabla_{\bx}\phi & =w(\bx,\phi(\bx,\lambda))\\
    \phi(\bx,0) & =z_{b}(\bx)\\
    \phi(\bx,1) & =\eta(\bx),
    \end{aligned}
    \right.
\]
has a solution, then \eqref{eq:def_phi} admits a steady state, which
leads to a stationary solution $(H,\tilde{\bu})$ of \eqref{eq:new_system}.
We note that the first equation above is satisfied for $\lambda\in(0,1)$.
This equation being quasilinear, in theory it is possible to solve it for example by starting with some initial data on a manifold of codimension one.
More precisely, the strategy is to take $\phi(\bx,\lambda)=(1-\lambda)z_{b}(\bx)+\lambda\eta(\bx)$
as initial data on a smooth manifold $\mathcal{M}\subset\R^{d}$ of
codimension one such that $\bu(\bx,\phi(\bx,\lambda))\cdot\boldsymbol{n}\neq0$
for $(\bx,\lambda)\in\partial\mathcal{M}\times I$ where $\boldsymbol{n}$
denotes the normal to $\mathcal{M}$. Then by the method of characteristics,
a solution exists provided the characteristics are defined globally.
\medskip

In the remaining of this section we characterize the stationary solutions of the new system \eqref{eq:new_system} in one dimension.
\begin{proposition} \label{prop:steady_new}
For $d=1$, the following properties are equivalent:
\begin{itemize}
    \item $(H,\tilde{u})$ is a stationary solution of \eqref{eq:new_system} with $H>0$ and $\tilde{u}>0$;
    \item there exists three functions $F:\bar{I}\to\R$, $G:\R\to\R$, and $Q:I\to\R$ with $F(\lambda)-G(x)>0$ for all $(x,\lambda)\in\tilde{\Omega}$, such that:
    \begin{align*}
        \tilde{u}(t,x,\lambda) & =\sqrt{F(\lambda)-G(x)} >0,\\
        H(t,x,\lambda) & =\frac{Q(\lambda)}{\sqrt{F(\lambda)-G(x)}}>0,\\
        z_{b}(x) & =\frac{G(x)}{2g}-\int_{0}^{1}\frac{Q(\lambda)}{\sqrt{F(\lambda)-G(x)}}\mathrm{d}\lambda.
    \end{align*}
\end{itemize}
\end{proposition}

\begin{proof}
For $d=1$, stationary solutions are characterized by
\begin{align*}
    \partial_{x}(H\tilde{u}) & =0,&
    \partial_{x}\left(\frac{\tilde{u}^{2}}{2}+g\int_{0}^{1}H\mathrm{d}\lambda+gz_{b}\right) & =0,\,
\end{align*}
which is equivalent to the existence of three functions $F,G,Q:I\to\R$ such that:
\begin{align*}
    H\tilde{u} & =Q(\lambda),&
    \tilde{u}^{2}+G(x) & =F(\lambda),
\end{align*}
where
\[
G(x) = 2g\int_{0}^{1}H\mathrm{d}\lambda+2gz_{b}.
\]
The equivalence then follows by using the positivity assumptions.
\end{proof}

The previous result is not very explicit as the topography is given in terms of the functions $F$, $G$, and $Q$. However, when $Q=F'$ this simplifies to a more explicit form, an example of which is shown in \cref{fig:sol_anal_sta}:

\begin{corollary}
\label{cor:sol_anal_sta}
Let $F:\bar{I}\to\R$ and $G:\R\to\R$ be such that $F(\lambda)-G(x)>0$ and $F'(\lambda)>0$ for all $(x,\lambda)\in\tilde{\Omega}$. Then the functions $H$ and $\tilde u$ defined by:
\begin{align*}
    H(t, x, \lambda) & = \frac{F^\prime (\lambda)}{\sqrt{F(\lambda)-G(x)}},\\
\tilde u(t, x, \lambda) & = \sqrt{F(\lambda)-G(x)},
\end{align*}
are stationary solutions of \eqref{eq:new_system}, provided the topography satisfies
\begin{equation} \label{eq:stat_hyp_zb}
    z_b (x) = \frac{G(x)}{2g} - 2\sqrt{F(1)-G(x)} + 2\sqrt{F(0)-G(x)}.
\end{equation}
Moreover the  quantities $h$, $\eta$, $\phi$ can be computed explicitly: 
\begin{align*}
    h &= \int_0^1 H\mathrm{d}\lambda = 2\sqrt{F(1)-G(x)} - 2\sqrt{F(0)-G(x)},\\
    \eta & =z_{b}+h=\frac{G(x)}{2g},\\
    \phi & =z_{b}+\int_{0}^{\lambda}H\mathrm{d}\lambda=\frac{G(x)}{2g}+2\sqrt{F(\lambda)-G(x)}-2\sqrt{F(1)-G(x)}.\\
\end{align*}
\end{corollary}

\begin{proof}
By taking $Q=F'$ in \cref{prop:steady_new}, the results follow from explicit calculations since
\[
    2\partial_\lambda\left(\sqrt{F(\lambda)-G(x)}\right) = \frac{F^\prime (\lambda)}{\sqrt{F(\lambda)-G(x)}}.
\]
\end{proof}

\begin{remark}
We note that the hypothesis \eqref{eq:stat_hyp_zb} on the bottom topography $z_b$ is not as constrained as it appears at first sight. Indeed, for any $F$ and $z_b$ satisfying:
\[
   \sup z_{b}\leq\sup_{z\leq F(0)}\left(\frac{z}{2g}-2\sqrt{F(1)-z}+2\sqrt{F(0)-z}\right).
\]
there exists a function $G$ satisfying the hypothesis \eqref{eq:stat_hyp_zb}.
In particular a sufficient condition for the existence of the function $G$ is given by:
\[
    \sup z_{b}\leq\frac{F(0)}{2g}-2\sqrt{F(1)-F(0)}.
\]
\end{remark}
\begin{figure}
\centering \includegraphics[width=10cm]{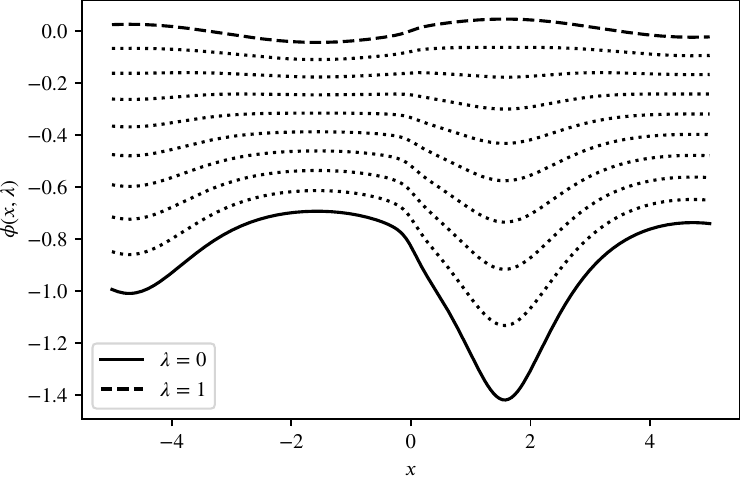}
\caption{Representation of the explicit steady solution given by \cref{cor:sol_anal_sta} for $F(\lambda)=1+\lambda$ and $G(x)=\frac{7}{10}\sin(x) + \frac{2}{10}\tanh(5x)$. The line $\lambda=0$ corresponds to the topography $z_b$, the line $\lambda=1$ to the free surface $\eta$, whereas the intermediate dotted lines correspond to values of $\lambda$ in between.}
\label{fig:sol_anal_sta}
\end{figure}

\subsection{Shallow water flows} \label{subsec:shallow_water}

Assuming that the horizontal velocity $\boldsymbol{u}$ does not depend on the vertical variable $z$, \emph{i.e.},
\[
    \boldsymbol{u}(t,\bx,z)=\boldsymbol{u}(t,\bx),
\]
then \eqref{eq:euler_3d} describes a columnar flow, \emph{i.e.}, without vertical motion, and moreover the horizontal motion is exactly governed by the shallow water system, also called the Saint-Venant equations.
In such a situation, \eqref{eq:def_phi}\subeq{1} becomes
\[
    \frac{\partial\phi}{\partial t}+\nabla_{\bx}\cdot\bigl((\phi-z_b)\bu\bigr) = 0,
\]
as results from the conservative form \eqref{eq:phi_conser}.
Considering the canonical initial condition \eqref{eq:phi0_canonic}
\[
    \phi_0(\bx,\lambda) = \lambda \eta_0(\bx) + (1-\lambda) z_b(\bx),
\]
one can check explicitly using \eqref{eq:ws_div} that the solution of \eqref{eq:def_phi} is given by
\[
    \phi(t,\bx,\lambda) = \lambda \eta(t,\bx) + (1-\lambda) z_b(\bx).
\]
In view of \eqref{eq:def_H_utilde}, $H$ and $\tilde{\bu}$  do not depend on $\lambda$:
\begin{align*}
    H(t,\bx,\lambda) & = \eta(t,\bx) - z_b(\bx), & \tilde{\bu}(t,\bx,\lambda) & = \bu(t,\bx),
\end{align*}
and \eqref{eq:new_system} reduces to the shallow water equations with trivial vertical dependency.
Hence with a suitable initial condition for $\phi$,  \eqref{eq:new_system} and the classical shallow water system share the same solutions.

\subsection{A flow with an horizontal velocity depending on the vertical coordinate} \label{subsec:flow_vorticity}
In this subsection, $p^{a}$
is not supposed constant. The following paragraph exhibits a situation where the solution $\phi$ of \eqref{eq:def_phi} is global but converges asymptotically for large time to a discontinuous function. This shows a situation where the change of variable can hardly be used for numerical implementation.

For $\eta_0\in\R$, the functions $\eta$, $\boldsymbol{u}$, $w$ defined by
\begin{align*}
    \eta(t,\bx) &= \eta_0, &
    \boldsymbol{u}(t,\bx,z) &= \bx\left(z-\frac{\eta_0}{2}\right), &
    w(t,\bx,z) &= z(\eta_0-z),
\end{align*}
are solutions of \eqref{eq:euler_3d} for $z_b=0$ and the atmospheric pressure given by 
\[
    p^a = -|\bx|^2 \frac{\eta_0^2}{8}.
\]
See \cref{fig:example_vorticity_uw} for a graphical representation of the solution $(\bu,w)$.
With the initial condition
\[
    \phi_0(\bx, \lambda) = \lambda \eta_0,
\]
the solution of \eqref{eq:def_phi} is given by
\begin{equation} \label{eq:example_vorticity_phi}
    \phi(t,\bx, \lambda) = \frac{\lambda \eta_0}{(1-\lambda )\mathrm{e}^{-t\eta_0} + \lambda}.
\end{equation}
Hence when $t$ becomes large, $H=\partial_\lambda\phi$ goes to zero when $\lambda\neq0$, $\phi$ evolves towards a discontinuous function asymptotically in time and the change of variable \eqref{eq:def_H_utilde} becomes difficult to implement. This difficulty is presented in \cref{fig:example_vorticity_phi}, which shows the evolution of $\phi$ for $t\in(0,6)$ and $\eta_0=1$. One can clearly see that $\phi$ evolves towards a discontinuous function in infinite time.

\begin{figure}[ht]
    \centering
    \includegraphics{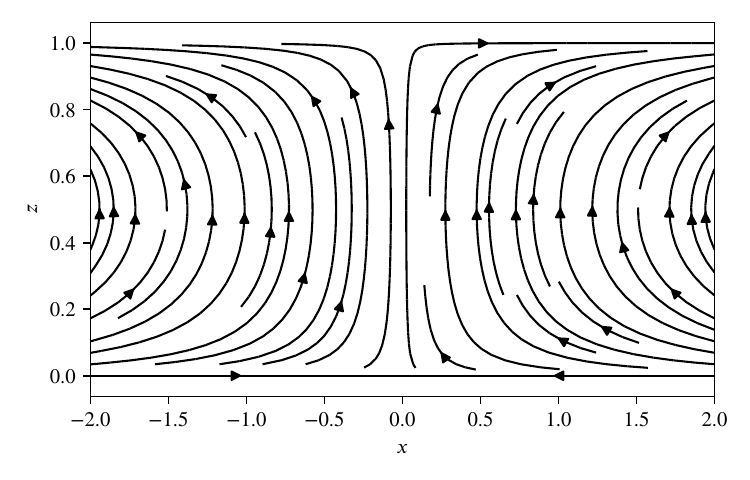}
    \caption{Explicit flow with vorticity proposed in \cref{subsec:flow_vorticity} with $\eta_0=1$.}
    \label{fig:example_vorticity_uw}
\end{figure}

\begin{figure}[ht]
    \centering
    \includegraphics{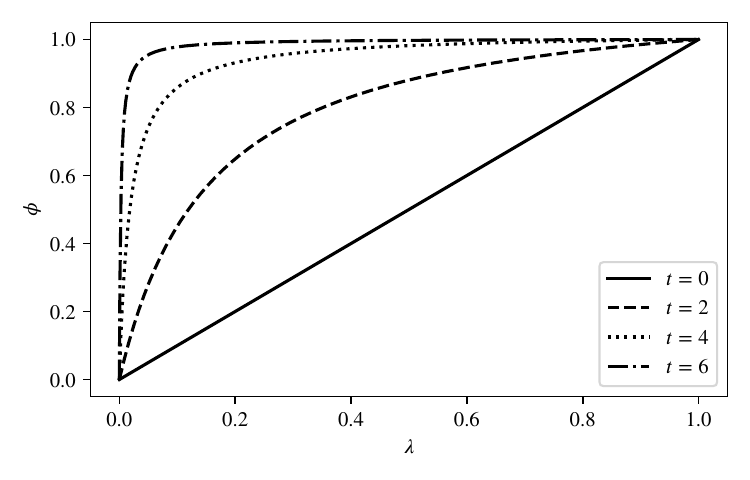}
    \caption{Graph of $\phi$ defined by \eqref{eq:example_vorticity_phi} for $\eta_0=1$ at times $t=0,2,4,6$.}
    \label{fig:example_vorticity_phi}
\end{figure}

\section{Spectrum and Riemann invariants}\label{sec:spectrum_riemann}

The semi-Lagrangian formulation \eqref{eq:new_system} has a chance for having an hyperbolic structure so it is quite natural to try to determine the eigenvalues and Riemann invariants of the associated operator. Due to the presence of the integral in $\lambda$ this is not a standard hyperbolic system, the matrix is replaced by an operator acting on functions of $\lambda$.

For simplicity we first consider the case $z_{b}=0$. The case of non-zero topography will be treated in \cref{subsec:avectopo}. For $d=2$, the system \eqref{eq:new_system} can be rewritten as
\begin{equation} \label{eq:new_system_vector}
    \left\{
    \begin{aligned}
    & \frac{\partial\boldsymbol{U}}{\partial t}+A_{1}(\boldsymbol{U})\frac{\partial\boldsymbol{U}}{\partial x}+A_{2}(\boldsymbol{U})\frac{\partial\boldsymbol{U}}{\partial y} = 0\\
    & \boldsymbol{U}(0) = (H_0, \tilde{u}_0, \tilde{v}_0)^T
    \end{aligned}
    \right.
\end{equation}
where $\boldsymbol{U}=(H,\tilde{\boldsymbol{u}})^{T}=(H,\tilde{u},\tilde{v})^T$ and
\begin{align}\label{def:mat_A}
    A_{1}(\boldsymbol{U}) & =\begin{pmatrix}\tilde{u} & H & 0\\
    g\int_{0}^{1}\cdot \,\mathrm{d}\lambda & \tilde{u} & 0\\
    0 & 0 & \tilde{u}
    \end{pmatrix}, & A_{2}(\boldsymbol{U}) & =\begin{pmatrix}\tilde{v} & 0 & H\\
    0 & \tilde{v} & 0\\
    g\int_{0}^{1}\cdot \,\mathrm{d}\lambda & 0 & \tilde{v}
    \end{pmatrix},
\end{align}
are non-local operators acting on functions in the variable $\lambda$.
The first aim is to study the hyperbolicity of \eqref{eq:new_system_vector}. Hence, we may consider that the variables $(t,\bx)$ are non-essential for the study of hyperbolicity of the operators $A_{1}(\boldsymbol{U})$ and $A_{2}(\boldsymbol{U})$ so we will often not write the dependence on $(t,\bx)$ explicitly for the sake of simplicity.
Let $\boldsymbol{\xi}=(\xi_{1},\xi_{2})\in\mathbb{R}^{2}$ such that
$\Vert\boldsymbol{\xi}\Vert=1$ and consider a linear combination
of the operators $A_{1}(\boldsymbol{U})$ and $A_{2}(\boldsymbol{U})$ given by \eqref{def:mat_A}: 
\begin{align*}
    A_{\boldsymbol{\xi}}(\boldsymbol{U}) &:=  \xi_{1}A_{1}(\boldsymbol{U})+\xi_{2}A_{2}(\boldsymbol{U}).
\end{align*}
System \eqref{eq:new_system_vector} is invariant by the rotation along the vertical axis, see~\cite[Chapter 5]{godlewski-book}.
By defining the following rotation matrix
\[
    R=\begin{pmatrix}
        1 & 0 & 0\\
        0 & \cos\theta & -\sin\theta\\
        0 & \sin\theta & \cos\theta
    \end{pmatrix},
\]
we have
\[
    R^{-1}A_{\boldsymbol{\xi}}(\boldsymbol{U})R=\begin{pmatrix}
        \boldsymbol{\xi}\cdot\tilde{\boldsymbol{u}} & H & 0\\
        g\int_{0}^{1}\cdot d\lambda & \boldsymbol{\xi}\cdot\tilde{\boldsymbol{u}} & 0\\
        0 & 0 & \boldsymbol{\xi}\cdot\tilde{\boldsymbol{u}}
    \end{pmatrix},
\]
with $\xi_1= \cos\theta, \xi_2=\sin\theta$. Hence the operators $R^{-1}A_{\boldsymbol{\xi}}(\boldsymbol{U})R$ and $A_{\boldsymbol{\xi}}(\boldsymbol{U})$ share the same spectrum.
Note that $R^{-1}A_{\boldsymbol{\xi}}(\boldsymbol{U})R=A_1(\boldsymbol{U}_{\boldsymbol{\xi}})$ where  $\boldsymbol{U}_{\boldsymbol{\xi}}=(H,\boldsymbol{\xi}\cdot\tilde{\boldsymbol{u}})^{T}$.
Moreover in view of the structure of $R^{-1}A_{\boldsymbol{\xi}}(\boldsymbol{U})R$, the third variable $\tilde v$ is decoupled from the first two variables, which motivates the study of the spectrum of the $2\times 2$ matrix $A_0(\tilde{\boldsymbol{U}})$ corresponding to the case $d=1$ given by
\[
    A_{0}(\tilde{\boldsymbol{U}}) =\begin{pmatrix}
        \tilde{u} & H\\
        g\int_{0}^{1}\cdot\,\mathrm{d}\lambda & \tilde{u}
    \end{pmatrix},
    \quad\text{where}\quad
    \tilde{\boldsymbol{U}} =\begin{pmatrix}
        H\\
    \tilde{u}
    \end{pmatrix}.
\]

Provided that $\tilde{u}\in L^{\infty}(I)$
and $H\in L^{\infty}(I)$, $A_0(\tilde{\boldsymbol{U}}):L^{2}(I)^{2}\to L^{2}(I)^{2}$ is a bounded operator. We  can  rewrite \eqref{eq:new_system} when $d=1$ in the  following form
\begin{equation}\label{eq:new_system_vector_1d}
    \frac{\partial\tilde{\boldsymbol{U}}}{\partial t}+A_{0}(\tilde{\boldsymbol{U}})\frac{\partial\tilde{\boldsymbol{U}}}{\partial x}=0,
\end{equation}
and the aim of the next section is to study the spectrum of the operator $A_0$.

\subsection{Characterization of the spectrum} \label{subsec:spectrum}
For a linear operator on an infinite dimensional space, say $A_0(\tilde{\boldsymbol{U}}):L^{2}(I)^{2}\to L^{2}(I)^{2}$,  the notion of eigenvalues becomes more inclusive and an analysis of the spectrum is essential. To simplify the notation, and as $\tilde{\boldsymbol{U}}$ is fixed, we
omit the argument $\tilde{\boldsymbol{U}}$ in $A_{0}(\tilde{\boldsymbol{U}})$
and simply use the notation $A_{0}$, except in certain situations where there is ambiguity. The spectrum of $A_0$ is defined as the set $\sigma(A_0)$ of all $c\in\mathbb{C}$ for which the operator $A_0-c\mathbf{I}$ is not invertible. We will use some standard definitions related to the decomposition of the spectrum (for the reader's convenience, these definitions are recalled in \cref{sec:appendix}).

\medskip
We denote by $C^{0,\alpha}(I)$ the space of Hölder continuous functions of index $\alpha$ over $I$, \emph{i.e.}
the set of functions $u:I\to\R$ such that for all $x,y\in I$, $|u(x)-u(y)| < K |x-y|^\alpha$ for some fixed constant $K>0$, called the Hölder constant.
The authors in \cite{teshukov1985} and \cite{teshukov1995} derived a formula satisfied by the eigenvalues of $A_0$ given by equation \eqref{eq:condition-eigenvalues} below. However, no regularity assumptions had been imposed and the computation of eigenelements was performed without precision on the regularity of the variables $H$ and  $\tilde{u}$. The following theorem is an extension of the results given in \cite{teshukov1985} into a complete characterization of the spectrum of the operator $A_0$ precisely when the velocity $\tilde{u}$ is $\frac{1}{4}$-Hölder continuous over $I$. This regularity for $\tilde{u}$ turns out to be the limiting case for the existence of at least two real and distinct eigenvalues outside the range of values of the velocity $\tilde{u}$ as presented in \cref{prop:limit_case}.

Let us state our main result on the characterization of the spectrum of $A_0$.

\begin{theorem}\label{thm:spectrum}

If $\tilde{u},H\in L^{\infty}(I)$ with $\inf_{I}H>0$, then the spectrum of $A_{0}:L^{2}(I)^{2}\to L^{2}(I)^{2}$
is
\[
\sigma(A_{0})=\tilde{u}[I]\cup\left\{ c\in\mathbb{C}\setminus\tilde{u}[I] : \int_{0}^{1}\frac{gH}{(c-\tilde{u})^{2}}\mathrm{d}\lambda=1\right\} ,
\]
with the following characterization
\begin{align*}
\sigma_{p}(A_{0}) & =\left\{ c\in\mathbb{C} : (c-\tilde{u})^{-1}\in L^{4}(I)\:\text{and}\int_{0}^{1}\frac{gH}{(c-\tilde{u})^{2}}\mathrm{d}\lambda=1\right\} \cup\left\{ c\in\tilde{u}[I] : \meas\bigl(\tilde{u}^{-1}(c)\bigr)>0\right\} ,\\
\sigma_{c}(A_{0}) & = \tilde{u}[I] \setminus \sigma_{p}(A_{0}) ,\\
\sigma_{r}(A_{0}) & =\emptyset,
\end{align*}
where the definitions of the point spectrum $\sigma_{p}(A_{0})$,
continuous spectrum $\sigma_{c}(A_{0})$, residual spectrum $\sigma_{r}(A_{0})$,
and essential range $\tilde{u}[I]$ are given in \cref{sec:appendix}.

In particular if $\tilde{u}\in C^{0,1/4}(\bar{I})$, then the essential
range $\tilde{u}[I]$ coincides with the image $\tilde{u}(\bar{I})$
and for $c\in\tilde{u}(\bar{I})$, $(c-\tilde{u})^{-1}\notin L^{4}(I)$,
so
\begin{align*}
\sigma_{p}(A_{0}) & =\left\{ c\in\mathbb{C}\setminus\tilde{u}(\bar{I}) : \int_{0}^{1}\frac{gH}{(c-\tilde{u})^{2}}\mathrm{d}\lambda=1\right\} \cup\left\{ c\in\tilde{u}(\bar{I})\,:\:\meas\bigl(\tilde{u}^{-1}(c)\bigr)>0\right\} ,\\
\sigma_{c}(A_{0}) & =\left\{ c\in\tilde{u}(\bar{I})\,:\:\meas\bigl(\tilde{u}^{-1}(c)\bigr)=0\right\} ,\\
\sigma_{r}(A_{0}) & =\emptyset.
\end{align*}

\end{theorem}

\begin{proof}
The proof is divided into three steps: injectivity, surjectivity, and the characterization of the range.
The particular case with Hölder condition is treated in a fourth point.
For the convenience of the reader, we provide in \cref{fig:spectrum_sets}
a representation of the different conditions involved in the proof.

\begin{figure}
\centering
\begin{tikzpicture}[scale=0.95]
	\fill[red!50] (-4.6,-3.5) rectangle (4.6,3.5);
    \draw[rounded corners, fill=black!20] (-4.4,-2.8) rectangle (4.4,2.8);
	\draw (-4.6,3.5) node[below right] {$\meas\bigl(\tilde{u}^{-1}(c)\bigr)>0$};
	\draw[fill=black!50] (0,0) ellipse (4 and 2.5);
	\draw[pattern=north east lines] (-1,0) circle (2);
	\draw[pattern=north west lines] (1,0) circle (2);
	\draw[pattern=vertical lines] (-1,0) circle (1);

	\begin{scope}
	    \clip (1,0) circle (2);
	   \fill[red, opacity=0.5] (-1,0) circle (2);
	\end{scope}

	\begin{scope}
		\clip (-1,0) circle (1);
		\fill[blue, opacity=0.5, even odd rule] (-1,0) circle (1) (1,0) circle (2);
	\end{scope}

	\draw[fill=black!20] (4.9,3.3) rectangle (5.5,2.7);
	\draw (5.5,3) node[right] {$\meas\bigl(\tilde{u}^{-1}(c)\bigr)=0$};

	\draw[fill=black!50] (4.9,2.3) rectangle (5.5,1.7);
	\draw (5.5,2) node[right] {$(c-\tilde{u})^{-1}\in L^2(I)$};

	\draw[pattern=north east lines] (4.9,1.3) rectangle (5.5,0.7);
	\draw (5.5,1) node[right] {$(c-\tilde{u})^{-1}\in L^4(I)$};

	\draw[pattern=vertical lines] (4.9,-0.3) rectangle (5.5,0.3);
	\draw (5.5,0) node[right] {$(c-\tilde{u})^{-1}\in L^\infty(I)$ \emph{i.e.} $c\notin\tilde{u}[I]$};

	\draw[pattern=north west lines] (4.9,-1.3) rectangle (5.5,-0.7);
	\draw (5.5,-1) node[right] {$\displaystyle\int_{0}^{1}\frac{gH}{(c-\tilde{u})^{2}}\mathrm{d}\lambda=1$};

	\draw[fill=red!50] (4.9,-1.7) rectangle (5.5,-2.3);
	\draw (5.5,-2) node[right] {$A_{0}-c\mathbf{I}$ not injective \emph{i.e.} $c\in\sigma_{p}(A_{0})$};

	\draw[fill=blue, opacity=0.5] (4.9,-2.7) rectangle (5.5,-3.3);
	\draw (5.5,-3) node[right] {$A_{0}-c\mathbf{I}$ bijective \emph{i.e.} $c\notin\sigma(A_{0})$};
\end{tikzpicture}
\caption{Representation of the various conditions on $c$ used in the proof of \cref{thm:spectrum}.}
\label{fig:spectrum_sets}
\end{figure}

\begin{enumerate}
\item \emph{Injectivity:} The first step is to determine when $A_{0}-c\mathbf{I}$
is not injective, \emph{i.e.} we determine the values  $c$ such
that there exists a non-zero solution $\bphi=(\phi_{1},\phi_{2})\in L^{2}(I)^{2}$
of 
\[
(A_{0}-c\mathbf{I})\,\bphi=0,
\]
or explicitly 
\begin{align}
H\phi_{2} & =(c-\tilde{u})\phi_{1},\label{eq:sys_eigen_sanstopo1}\\
g\int_{0}^{1}\phi_{1}\,\mathrm{d}\lambda & =(c-\tilde{u})\phi_{2},\label{eq:sys_eigen_sanstopo2}
\end{align}
which yields in particular
\[
gH\int_{0}^{1}\phi_{1}\,\mathrm{d}\lambda=(c-\tilde{u})H\phi_{2}=(c-\tilde{u})^{2}\phi_{1}.
\]
One has to distinguish the cases when the inverse image $\tilde{u}^{-1}(c)$  of $\{c\}$ has zero or nonzero measure.
\begin{enumerate}
\item If $\tilde{u}^{-1}(c)$ has nonzero measure, then there exists a nontrivial
function $\phi_{1}\in L^{2}(I)$ with $\phi_{1}=0$ on $I\setminus\tilde{u}^{-1}(c)$
and such that
\[
\int_{0}^{1}\phi_{1}\,\mathrm{d}\lambda=\int_{\tilde{u}^{-1}(c)}\phi_{1}\,\mathrm{d}\lambda=0.
\]
One can check easily that such a $\phi_{1}$ is a solution of \eqref{eq:sys_eigen_sanstopo1}
and \eqref{eq:sys_eigen_sanstopo2} together with $\phi_{2}=0$. Thus
if $\tilde{u}^{-1}(c)$ has nonzero measure, then $A_{0}-c\mathbf{I}$ is not injective, and these values $c$ are in the point spectrum $\sigma_{p}(A_{0})$.
\item If $\tilde{u}^{-1}(c)$ has zero measure, the solution of \eqref{eq:sys_eigen_sanstopo1}-\eqref{eq:sys_eigen_sanstopo2} is given by
\begin{align}
\phi_{1} & =\frac{gH}{(c-\tilde{u})^{2}}\int_{0}^{1}\phi_{1}\,\mathrm{d}\lambda,\label{eq:sol_phi_sanstopo1}\\
\phi_{2} & =\frac{g}{c-\tilde{u}}\int_{0}^{1}\phi_{1}\,\mathrm{d}\lambda,\label{eq:sol_phi_sanstopo2}
\end{align}
and it remains to check whether $\bphi=(\phi_{1},\phi_{2})$ is in
$L^{2}(I)^{2}$.

Obviously $\bphi$ is trivial if and only if $\int_{0}^{1}\phi_{1}\,\mathrm{d}\lambda=0$,
so without lost of generality $\phi_{1}$ can be normalized such that
$\int_{0}^{1}\phi_{1}\,\mathrm{d}\lambda=1$. Then equations \eqref{eq:sol_phi_sanstopo1}-\eqref{eq:sol_phi_sanstopo2} reduce to 
\begin{align}
\phi_{1} & =\frac{gH}{(c-\tilde{u})^{2}},\label{eq:redsol_phi_sanstopo1}\\
\phi_{2} & =\frac{g}{c-\tilde{u}}.\label{eq:redsol_phi_sanstopo2}
\end{align}
We now have to distinguish the cases $(c-\tilde{u})^{-1}\notin L^{4}(I)$
and $(c-\tilde{u})^{-1}\in L^{4}(I)$:
\begin{itemize}
\item If $(c-\tilde{u})^{-1}\notin L^{4}(I)$, then $(c-\tilde{u})^{-2}\notin L^{2}(I)$
and there is no non-trivial solution for $\bphi$. Thus if $\tilde{u}^{-1}(c)$
has zero measure and $(c-\tilde{u})^{-1}\notin L^{4}(I)$, then $A_{0}-c\mathbf{I}$
is injective, so such values $c$ are not in the point spectrum $\sigma_{p}(A_{0})$.

\item On the other hand, if $(c-\tilde{u})^{-1}\in L^{4}(I)$, then $(c-\tilde{u})^{-2}\in L^{2}(I)$
and $(c-\tilde{u})^{-1}\in L^{2}(I)$. Hence,
equations \eqref{eq:sys_eigen_sanstopo1} and \eqref{eq:sys_eigen_sanstopo2}
admit a solution $\bphi\in L^{2}(I)^{2}$. As explained
above the solution $\bphi$ is not trivial as long as $\int_{0}^{1}\phi_{1}\,\mathrm{d}\lambda\neq0$,
so integrating \eqref{eq:sol_phi_sanstopo1} on $\lambda$ implies
the following characterization for having a non-trivial solution $\bphi$\,:
\begin{equation}
\int_{0}^{1}\frac{gH}{(c-\tilde{u})^{2}}\mathrm{d}\lambda=1.\label{eq:condition-eigenvalues}
\end{equation}
\end{itemize}
\end{enumerate}
This terminates the characterization of the injectivity of $A_{0}-c\mathbf{I}$,
hence the claimed characterization of the point-spectrum.
\item \emph{Surjectivity:} One has to determine the range of $A_{0}-c\mathbf{I}$,
\emph{i.e.} determining for which $\boldsymbol{b}=(b_{1},b_{2})\in L^2(I)^2$ the system 
\begin{align*}
H\phi_{2} & =(c-\tilde{u})\phi_{1}+b_{1},\\
g\int_{0}^{1}\phi_{1}\,\mathrm{d}\lambda & =(c-\tilde{u})\phi_{2}+b_{2},
\end{align*}
has a solution $\bphi=(\phi_{1},\phi_{2})\in L^{2}(I)^{2}$.
Obviously one has to consider only the case where $A_{0}-c\mathbf{I}$ is injective, so $\tilde{u}^{-1}(c)$ has zero measure.
One has
\[
gH\int_{0}^{1}\phi_{1}\,\mathrm{d}\lambda=(c-\tilde{u})H\phi_{2}+Hb_{2}=(c-\tilde{u})^{2}\phi_{1}+(c-\tilde{u})b_{1}+Hb_{2},
\]
therefore
\begin{align}
\phi_{1} & =\frac{H}{(c-\tilde{u})^{2}}\left(g\int_{0}^{1}\phi_{1}\,\mathrm{d}\lambda-b_{2}\right)-\frac{b_{1}}{c-\tilde{u}},\label{eq:sol_phi_b1}\\
\phi_{2} & =\frac{1}{c-\tilde{u}}\left(g\int_{0}^{1}\phi_{1}\,\mathrm{d}\lambda-b_{2}\right).\label{eq:sol_phi_b2} 
\end{align}

We now have to distinguish whether $(c-\tilde{u})^{-1}\notin L^{4}(I)$
or $(c-\tilde{u})^{-1}\in L^{4}(I)$.
\begin{enumerate}
\item If $(c-\tilde{u})^{-1}\notin L^{4}(I)$, then as before $(c-\tilde{u})^{-2}\notin L^{2}(I)$
and thus, taking for example $b_{2}=\lambda$ and $b_{1}=0$ ensures
that $\phi_{1}\notin L^{2}(I)$ since $\inf_{I}H>0$. Therefore if $(c-\tilde{u})^{-1}\notin L^{4}(I)$,
then $A_{0}-c\mathbf{I}$ is not surjective.
\item If $(c-\tilde{u})^{-1}\in L^{4}(I)$, then $(c-\tilde{u})^{-2}\in L^{2}(I)$
and $(c-\tilde{u})^{-1}\in L^{2}(I)$, thus we have
\[
\int_{0}^{1}\phi_{1}\,\mathrm{d}\lambda=\int_{0}^{1}\frac{gH}{(c-\tilde{u})^{2}}\mathrm{d}\lambda\,\int_{0}^{1}\phi_{1}\,\mathrm{d}\lambda-\int_{0}^{1}\left(\frac{b_{1}}{c-\tilde{u}}+\frac{Hb_{2}}{(c-\tilde{u})^{2}}\right),
\]
so
\[
\kappa\int_{0}^{1}\phi_{1}\,\mathrm{d}\lambda=\int_{0}^{1}\left(\frac{b_{1}}{c-\tilde{u}}+\frac{Hb_{2}}{(c-\tilde{u})^{2}}\right),\qquad\text{where}\qquad\kappa=\int_{0}^{1}\frac{gH}{(c-\tilde{u})^{2}}\mathrm{d}\lambda-1.
\]
Since we are only interested in the case where $A_{0}-c\mathbf{I}$
is injective, condition \eqref{eq:condition-eigenvalues} holds so
$\kappa\neq0$, hence $\int_{0}^{1}\phi_{1}\,\mathrm{d}\lambda$ is
well-defined in terms of $b_{1}$and $b_{2}$. We have that equations
\eqref{eq:sol_phi_b1} and \eqref{eq:sol_phi_b2} provide a solution
$\bphi=(\phi_{1},\phi_{2})\in L^{2}(I)^{2}$ for all $\boldsymbol{b}=(b_{1},b_{2})\in L^{2}(I)^{2}$
if and only if $(c-\tilde{u})^{-1}\in L^{\infty}(I)$. We note that
as explained in \cref{sec:appendix}, $(c-\tilde{u})^{-1}\in L^{\infty}(I)$
if and only if $c$ does not belong to the essential range $\tilde{u}[I]$
of $\tilde{u}$.
\end{enumerate}
Consequently, $A_{0}-c\mathbf{I}$ is injective but not surjective
when $\tilde{u}^{-1}(c)$ has zero measure and $c\in\tilde{u}(\bar{I})$.

\item \emph{Density of the range:} It remains to characterize the range
of $A_0-c\mathbf{I}$ when $\tilde{u}^{-1}(c)$ has zero measure.
We will prove that the
range is dense in $L^{2}(I)^{2}$, so that $\sigma_{r}(A_{0})=\emptyset$,
which finishes to characterize $\sigma_{c}(A_{0})$ and $\sigma_{r}(A_{0})$
as stated.
Since $L^{\infty}(I)^{2}$ is dense in $L^{2}(I)^{2}$
it suffices to show that for any $\boldsymbol{b}\in L^{\infty}(I)^{2}$,
there exists a sequence $\bphi_{n}\in L^{2}(I)^{2}$ such that $\boldsymbol{b}_{n}=(A_{0}-c\mathbf{I})\bphi_{n}$
converges to $\boldsymbol{b}$ in $L^{2}(I)^{2}$.
Let us define the following characteristic function
\[
\chi_{n}(\lambda)=\begin{cases}
1 & \text{if }\left|\tilde{u}(\lambda)-c\right|>\frac{1}{n},\\
0 & \text{otherwise},
\end{cases}
\]
and note that since $\tilde{u}^{-1}(c)$ has zero measure, $\chi_n$ converges pointwise to $1$ almost everywhere as $n\to\infty$.
Let $\boldsymbol{b}=(b_1,b_2)\in L^{\infty}(I)^{2}$ and consider $\bphi_{n}=(\phi_{1n},\phi_{2n})\in L^{2}(I)^{2}$
defined by
\begin{align*}
\phi_{1n} & =\frac{\chi_{n}H}{(c-\tilde{u})^{2}}\left[g\theta_{n}-b_{2}\right]-\frac{\chi_{n}b_{1}}{c-\tilde{u}},\\
\phi_{2n} & =\frac{\chi_{n}}{c-\tilde{u}}\left[g\theta_{n}-b_{2}\right],
\end{align*}
where $\theta_{n}$ will be defined later.
One has
\[
\int_{0}^{1}\phi_{1n}\,\mathrm{d}\lambda=\alpha_{n}-\beta_{n},
\]
where
\begin{align*}
\alpha_{n} & =\int_{0}^{1}\frac{g\chi_{n}\theta_{n}H}{(c-\tilde{u})^{2}}\mathrm{d}\lambda, & \beta_{n} & =\int_{0}^{1}\chi_{n}\left(\frac{b_{1}}{c-\tilde{u}}+\frac{Hb_{2}}{(c-\tilde{u})^{2}}\right)\mathrm{d}\lambda.
\end{align*}
Therefore $\boldsymbol{b}_{n}=(A_{0}-c\mathbf{I})\bphi_{n}$ is explicitly
given by
\begin{align*}
b_{n1} & =\chi_{n}b_{1},\\
b_{n2} & =\chi_{n}b_{2}+g\left[\alpha_{n}-\beta_{n}-\chi_{n}\theta_{n}\right].
\end{align*}
Since $\chi_{n}\boldsymbol{b}$ converges pointwise to $\boldsymbol{b}$,
the Lebesgue dominated convergence theorem implies that $\chi_{n}\boldsymbol{b}$
converges to $\boldsymbol{b}$ in $L^{2}(I)^{2}$. Now let us 
prove that the remainder $R_{n}=\alpha_{n}-\beta_{n}-\chi_{n}\theta_{n}$ converges
to zero in $L^{2}(I)$. To this end, we distinguish two cases: (a)
$(c-\tilde{u})^{-1}\in L^{4}$ and condition \eqref{eq:condition-eigenvalues}
does not hold, (b) $(c-\tilde{u})^{-1}\in L^{2}(I)\setminus L^{4}(I)$
and (c) $(c-\tilde{u})^{-1}\notin L^{2}(I)$.
\begin{enumerate}
\item If $(c-\tilde{u})^{-1}\in L^{4}(I)$, we choose
\begin{align*}
\theta_{n} & =\beta_{n}\kappa_{n}^{-1}, & \kappa_{n} & =\int_{0}^{1}\frac{g\chi_{n}H}{(c-\tilde{u})^{2}}\mathrm{d}\lambda-1,
\end{align*}
so that
\[
R_{n}=\alpha_{n}-\beta_{n}-\chi_{n}\theta_{n}=(1-\chi_{n})\theta_{n}.
\]
Since $(c-\tilde{u})^{-1}\in L^{2}(I)$, $\kappa_{n}$ converges to
\[
\kappa=\int_{0}^{1}\frac{gH}{(c-\tilde{u})^{2}}\mathrm{d}\lambda-1,
\]
which is nonzero since \eqref{eq:condition-eigenvalues} does not
hold, we deduce that for $n$ large enough $\kappa_{n}\neq0$ and
$\kappa_{n}^{-1}$ is bounded. Moreover, since $(c-\tilde{u})^{-2}\in L^{1}(I)$,
and $\boldsymbol{b}\in L^{\infty}(I)^{2}$, $\beta_{n}$ is bounded. Therefore
$\theta_{n}$ is bounded, hence $R_{n}$ is uniformly bounded in $L^{2}(I)$
and by the dominated convergence theorem, $R_{n}$ converges to zero
in $L^{2}(I)$.
\item If $(c-\tilde{u})^{-1}\in L^{2}(I)$ but $(c-\tilde{u})^{-1}\notin L^{4}(I)$,
we choose
\[
\theta_{n}=\frac{\beta_{n}C_{n}^{-1}}{(c-\tilde{u})^{2}},\qquad\text{where}\qquad C_{n}=\int_{0}^{1}\frac{\chi_{n}gH}{(c-\tilde{u})^{4}}\mathrm{d}\lambda,
\]
so that $\alpha_{n}=\beta_{n}$ and moreover since $\inf_{I}H>0$, we have $C_{n}\to\infty$
in the limit $n\to\infty$. Since $(c-\tilde{u})^{-1}\in L^{2}(I)$
and $\boldsymbol{b}\in L^{\infty}(I)^{2}$, $\beta_{n}$ is bounded. Finally,
we have
\[
\Vert\chi_{n}\theta_{n}\Vert_{L^{2}(I)}^{2}=\beta_{n}^{2}C_{n}^{-2}\int_{0}^{1}\frac{\chi_{n}}{(c-\tilde{u})^{4}}\mathrm{d}\lambda\leq\frac{\beta_{n}^{2}C_{n}^{-1}}{g\inf_{I}H},
\]
which converges to zero as $n\to\infty$. This finishes to prove that
$R_{n}$ converges to zero in $L^{2}(I)$.
\item If $(c-\tilde{u})^{-1}\notin L^{2}(I)$ , we choose
\begin{align*}
\theta_{n} & =\beta_{n}\kappa_{n}^{-1}, & \kappa_{n} & =\int_{0}^{1}\frac{g\chi_{n}H}{(c-\tilde{u})^{2}}\mathrm{d}\lambda-1,
\end{align*}
as in case (a) so that
\[
R_{n}=\alpha_{n}-\beta_{n}-\chi_{n}\theta_{n}=(1-\chi_{n})\theta_{n}.
\]
Since $(c-\tilde{u})^{-2}\notin L^{1}(I)$, $\kappa_{n}$ is expected
to diverge in the limit $n\to\infty$. More precisely, we have $\kappa_{n}\geq g\inf_{I}H\,C_{n}-1$,
where
\[
C_{n}=\int_{0}^{1}\frac{\chi_{n}}{(c-\tilde{u})^{2}}\mathrm{d}\lambda,
\]
and $C_{n}\to\infty$ as $n\to\infty$ since $(c-\tilde{u})^{-2}\notin L^{1}(I)$.
Therefore, for $n$ large enough $\kappa_{n}>0$ and $\theta_{n}$
is well-defined. Moreover, since $H$, $b_{1}$ and $b_{2}$ are uniformly
bounded, one has for $n$ large enough
\begin{align*}
|\theta_{n}| & \leq\kappa_{n}^{-1}\left(\|b_{1}\|_{\infty}\int_{0}^{1}\frac{\chi_{n}}{c-\tilde{u}}\mathrm{d}\lambda+\|H\|_{\infty}\|b_{2}\|_{\infty}\int_{0}^{1}\frac{\chi_{n}}{(c-\tilde{u})^{2}}\mathrm{d}\lambda\right)\\
 & \leq\frac{1}{g\inf_{I}H\,C_{n}-1}\left(\|b_{1}\|_{\infty}(C_{n}+1)+\|H\|_{\infty}\|b_{2}\|_{\infty}C_{n}\right)\\
 & \lesssim\frac{\|b_{1}\|_{\infty}+\|H\|_{\infty}\|b_{2}\|_{\infty}}{g\inf_{I}H},
\end{align*}
so $\theta_{n}$ is a bounded sequence. Therefore, as before this
finishes to prove that $R_{n}$ converges to zero in $L^{2}(I)$.
\end{enumerate}

\item \emph{Particular case:} Finally, we consider the special case $\tilde{u}\in C^{0,1/4}(\bar{I})$. Since $\tilde{u}$ is continuous, the essential range of $\tilde{u}$ coincides with the image $\tilde{u}(\bar{I})$ and given $c\in\tilde{u}(\bar{I})$ there exists $\lambda_{0}\in\bar{I}$ such that $\tilde{u}(\lambda_{0})=c$. Moreover, since $\tilde{u}\in C^{0,1/4}(\bar{I})$, there exists a constant $K>0$ such that for all $\lambda\in\bar{I}$, 
\[
\left|c-\tilde{u}(\lambda)\right|=\left|\tilde{u}(\lambda)-\tilde{u}(\lambda_{0})\right|\leq K\left|\lambda-\lambda_{0}\right|^{1/4}.
\]
Therefore 
\[
\frac{1}{\left|c-\tilde{u}(\lambda)\right|}\geq\frac{1}{K\left|\lambda-\lambda_{0}\right|^{1/4}},
\]
so $(c-\tilde{u})^{-1}\notin L^{4}(I)$.

\end{enumerate}
\end{proof}

From this result, we can easily deduce the following alternative characterization of the spectrum which is useful to distinguish between two types of elements in the spectrum: the elements that satisfy the integral condition and the range of values of the velocity $\tilde{u}$ on the interval $I$.
\begin{corollary}\label{cor:spectrum}
Under the hypotheses of \cref{thm:spectrum}, the spectrum of $A_{0}:L^{2}(I)^{2}\to L^{2}(I)^{2}$
is characterized by
\begin{align*}
\sigma_{d}(A_{0}) & =\left\{ c\in\mathbb{C}\setminus\tilde{u}(I) : \int_{0}^{1}\frac{gH}{(c-\tilde{u})^{2}}\mathrm{d}\lambda=1\right\} ,\\
\sigma_{e}(A_{0}) & =\tilde{u}(I),
\end{align*}
where  $\sigma_{d}(A_{0})$ is the discrete spectrum and  $\sigma_{e}(A_{0})$ is the essential spectrum defined in \cref{sec:appendix}.
\end{corollary}

\begin{proof}
The operator $A_0$ is a compact perturbation of the operator $A_0$ with $g=0$. More precisely, we have $A_0=D_0 + K_0$, where
\begin{align*}
    D_0 &= \begin{pmatrix}
        \tilde{u} & H\\
        0 & \tilde{u}
    \end{pmatrix},&
    K_0 &= \begin{pmatrix}
        0 & 0\\
        g\int_{0}^{1}\cdot\,\mathrm{d}\lambda & 0
    \end{pmatrix},
\end{align*}
and $K_0$ is compact (even finite-rank).
Using the fact that the essential spectrum is invariant under compact perturbations, we directly deduce the corollary using \cref{thm:spectrum}.
\end{proof}

\begin{remark}

The above results can probably be generalized to $A_0:L^p(I)^2\mapsto L^p(I)^2$ for $p\geq1$, however, some parts of the proof (like 3. (b)) require some adaptation.

\end{remark}

\begin{remark} \label{rem:cond_dom_reel}
    The integral condition \eqref{eq:condition-eigenvalues} can be rewritten as follows in the original variables:
    \[
        1=\int_{0}^{1}\frac{gH}{(c-\tilde{u})^{2}}\mathrm{d}\lambda=\int_{z_b}^{\eta}\frac{g}{(c-u)^{2}}\mathrm{d}z.
    \]
\end{remark}

\begin{remark} \label{rem:spectrum_saint_venant}
In the shallow water regime (see \cref{subsec:shallow_water}), \emph{i.e.}, when $\tilde{u}$ is independent of $\lambda$, then the spectrum reduces to
\[
    \sigma(A_{0}) = \left\{ \tilde{u} + \sqrt{gh}, \tilde{u} - \sqrt{gh}, \tilde{u} \right\},
    \qquad\text{where}\qquad
    h = \int_0^1 H \mathrm{d}\lambda = \eta - z_b.
\]
The first two eigenvalues $\tilde{u} \pm \sqrt{gh}$ are the usual eigenvalues of the Saint-Venant system, whereas $\tilde{u}$ is a spurious eigenvalue coming from the fact that adding a dependency in $\lambda$ in Saint-Venant, which is expressed through a transport equation with velocity $u$ for $\phi$, is artificial.
\end{remark}

We have the following localization result of the spectrum, which is a reminiscence of Howard's semi-circle theorem \cite{Howard1961}.
\begin{proposition} \label{prop:spectrum_localization}
    Under the assumptions of \cref{thm:spectrum}, $\sigma(A_0)$ is contained in the union of the following three sets represented on \cref{fig:spectrum_localization}:
    \begin{align}\label{eq:def_J_pm}
        J_{-} &= \bigl[\tilde{u}_{-}-\sqrt{gh},\tilde{u}_{-}\bigr), &
        J_{+} &= \bigl(\tilde{u}_{+},\tilde{u}_{+}+\sqrt{gh}\bigr].
    \end{align}
    
    \begin{equation}
        R=\Bigl\{ z\in\mathbb{C}:\tilde{u}_{-}\leq\Re z\leq\tilde{u}_{+}\text{ and }|\Im z|\leq\sqrt{gh}\text{ and } \Bigl|z-\frac{\tilde{u}_{-}+\tilde{u}_{+}}{2}\Bigr|\leq\frac{\tilde{u}_{+}-\tilde{u}_{-}}{2} \Bigr\},\label{eq:def_R}
    \end{equation}
    
    where
    \begin{align*}
        \tilde{u}_{-} &= \inf_I \tilde{u}, &
        \tilde{u}_{+} &= \sup_I \tilde{u}, &
        h = \int_{0}^{1}H\mathrm{d}\lambda.
    \end{align*}
\end{proposition}

\begin{figure}
\centering
\begin{tikzpicture}[scale=1.2]
	\draw[thick,->] (-4.5,0) -- (4.5,0) node[right] {$\Re c$};
	\draw[thick,->] (0,-2.) -- (0,2) node[above] {$\Im c$};
	\draw (0,1.5) node[right] {$+\sqrt{gh}$};
	\draw (0,-1.5) node[right] {$-\sqrt{gh}$};
    \begin{scope}
      \clip (0,0) circle (2.5);
	  \fill[nearly transparent](-2.5,-1.5) rectangle (2.5,1.5);
    \end{scope}
	\draw (-2.1,1.5) node[below right] {$R$};
	\draw (0,0) circle (2.5);
	\fill[nearly transparent](-4,-0.05) rectangle (-2.5,0.05);
	\draw (-3.25,0) node[below] {$J_-$};
	\fill[nearly transparent](4,-0.05) rectangle (2.5,0.05);
	\draw (3.25,0) node[below] {$J_+$};
	\draw[fill] (-2.5,0) circle (2pt);
	\draw (-2.5,0) node[below] {$\tilde{u}_-$};
	\draw[fill] (2.5,0) circle (2pt);
	\draw (2.5,0) node[below] {$\tilde{u}_+$};
	\draw[line width=2pt] (-2.5,0) -- (2.5,0);
	\draw (0,0) node[below right] {$\tilde{u}[I]$};
\end{tikzpicture}
\caption{ The spectrum $\sigma(A)$ is included in the rectangle of height $2\sqrt{gh}$ bounded between $\tilde{u}_{-}$ and $\tilde{u}_{+}$ clipped by the circle of radius $(\tilde{u}_{+}-\tilde{u}_{-})/2$ centered at $(\tilde{u}_{-}+\tilde{u}_{+})/2$ $R$ and intervals $J_\pm$ defined in \cref{prop:spectrum_localization}. The continuous part of the spectrum $\tilde{u}[I]$ is included in the line between $\tilde{u}_{-}$ and $\tilde{u}_{+}$.}
\label{fig:spectrum_localization}
\end{figure}

\begin{proof}
It suffices to prove that $\sigma_{d}(A_{0})$ is included in the
claimed set, so to consider the values $c$ satisfying the integral condition \eqref{eq:condition-eigenvalues}.
Therefore this reduces to study the function $F$ defined on
$\mathbb{C}\setminus[\tilde{u}_{-},\tilde{u}_{+}]$ by: 
\begin{equation}
F(c)=\int_{0}^{1}\frac{gH}{(c-\tilde{u})^{2}}\mathrm{d}\lambda.\label{eq:F_c}
\end{equation}
The proof is divided into three steps.
\begin{enumerate}
    \item On the one hand, since
    \[
        \Im F(c)=-2\int_{0}^{1}\frac{gH(\Re c-\tilde{u})\Im c}{|c-\tilde{u}|^{4}}\mathrm{d}\lambda,
    \]
    if $\Re c>\tilde{u}_{+}$ or $\Re c<\tilde{u}_{-}$, then
    $(\Re c-\tilde{u})$ is either strictly positive or strictly negative
    on $I$, so the only way to make the integral $\Im F(c)$ zero, is
    that $\Im c=0$. This proves that 
    \[
        \sigma(A_{0})\subset\left\{ z\in\mathbb{C}:\tilde{u}_{-}\leq\Re z\leq\tilde{u}_{+}\right\} \cup\R.
    \]
    
    \item On the other hand, if $\dist(c,[\tilde{u}_{-},\tilde{u}_{+}])>\sqrt{gh}$, then for almost any
    $\lambda\in I$, $|c-\tilde{u}|>\sqrt{gh}$ and therefore
    \[
    |F(c)|=\int_{0}^{1}\frac{gH}{|c-\tilde{u}|^{2}}\mathrm{d}\lambda<\frac{1}{h}\int_{0}^{1}H\mathrm{d}\lambda=1,
    \]
    so $c$ cannot be an eigenvalue. Therefore, we have shown that
    \[
    \sigma(A_{0})\subset\left\{ z\in\mathbb{C} : \dist(z,[\tilde{u}_{-},\tilde{u}_{+}])\leq\sqrt{gh}\right\}.
    \]
    
    \item Finally, we assume that $\Im c\neq0$. An explicit calculation shows
    that for any $\beta\in\mathbb{R}$
    \[
        \Re F(c)+\frac{\Im F(c)}{\Im c}\left(\Re c-\beta\right)=G(c,\beta),
    \]
    where
    \[
        G(c,\beta)=\int_{0}^{1}\frac{gH}{|c-\tilde{u}|^{4}}\left((\tilde{u}-\beta)^{2}-|c-\beta|^{2}\right)\mathrm{d}\lambda.
    \]
    Therefore, $F(c)=1$ if and only if $G(c,\beta)=1$ for all $\beta\in\mathbb{R}$.
    Since $|\tilde{u}-\beta|\leq\max\{|\tilde{u}_{-}-\beta|,|\tilde{u}_{+}-\beta|\}$,
    we deduce the following property: if for some $\beta\in\mathbb{R}$, $|c-\beta|>\max\{|\tilde{u}_{-}-\beta|,|\tilde{u}_{+}-\beta|\}$,
    then $G(c,\beta)\leq0$ and $c$ cannot be an eigenvalue. One can
    check that this condition defines the largest complement of a disk for the optimal value $\beta=(\tilde{u}_{-}+\tilde{u}_{+})/2$
    and for this value the condition is exactly $|c-\beta|>(\tilde{u}_{+}-\tilde{u}_{-})/2$.
    Therefore, we showed that
    \[
        \sigma(A_{0})\cap\left(\mathbb{C}\setminus\mathbb{R}\right)\subset\Bigl\{ z\in\mathbb{C}:\Bigl|z-\frac{\tilde{u}_{-}+\tilde{u}_{+}}{2}\Bigr|\leq\frac{\tilde{u}_{+}-\tilde{u}_{-}}{2}\Bigr\},
    \]
    and the proof is complete.
\end{enumerate}
\end{proof}

\subsection{Limiting cases}\label{subsec:limiting_cases}
Finding whether the values $c\in \sigma_d(A_0)$ are real or complex is essential to determine the hyperbolicity conditions for the associated system \eqref{eq:new_system_vector_1d}. As noted in \cite{     chesnokov2009,chesnokov2017,teshukov2004,teshukov1997,teshukov1995} for regular velocity profiles there exists at least two real eigenvalues of the operator $A_0$ outside the range of values of the velocity $\tilde{u}$. In \cite{chesnokov2017}, generalized hyperbolicity conditions for system \eqref{eq:new_system_vector_1d} with a monotonic velocity profile are formulated using the limit values of \eqref{eq:condition-eigenvalues} in the upper and lower complex half-planes. 
In the general case, as shown in this subsection, these two real solutions can only exist under certain regularity assumptions on the velocity $\tilde{u}(\lambda)$.

We have the following result on the minimum regularity criterion for the existence of these two real eigenvalues.

\begin{proposition}\label{prop:limit_case}
Under the assumptions of \cref{thm:spectrum}, if either:
\begin{enumerate}
\item $\tilde{u}\in C^{0,1/2}(\bar{I})$, or
\item $\tilde{u}\in C^{0,1/4}(\bar{I})$ with Hölder constant $K>0$ satisfying
$K<\sqrt{2g\displaystyle\inf_I H}$,
\end{enumerate}
there exists exactly two real eigenvalues $c_{\pm}\in J_{\pm}$,
where the intervals $J_{\pm}$ are defined in \eqref{eq:def_J_pm},
such that
\[
    \sigma(A_0)\cap\mathbb{R}=\left\{ c_{-},c_{+}\right\} \cup \tilde{u}(\bar{I}), 
    \quad\text{or}\quad
    \sigma_{d}(A_0)\cap\mathbb{R}=\left\{ c_{-},c_{+}\right\}.
\]
\end{proposition}

\begin{proof}
In view of \eqref{eq:condition-eigenvalues}, this reduces to
study  the function $F$ defined on $J_{+}\cup J_{-}$ by \eqref{eq:F_c}. 
One can check that the function $F$ is continuous on $J_{\pm}$.
The first step is to show that
\begin{align*}
    \lim_{c\to \tilde{u}_{-},c<\tilde{u}_{-}}F(c) > 1 \quad \textrm{ and } \quad \lim_{c\to \tilde{u}_{+},c>\tilde{u}_{+}}F(c) > 1.
\end{align*}
Since $\tilde{u}$ is continuous, there exists $\lambda_{\pm}\in \bar{I}$
such that $\tilde{u}(\lambda_{\pm})=\tilde{u}_{\pm}$.
\begin{enumerate}
\item{If $\tilde{u}$ is $\frac{\ensuremath{1}}{2}$-Hölder continuous,
there exists $K>0$ such that for all $\lambda\in \bar{I}$, 
\[
\left|\tilde{u}(\lambda)-\tilde{u}_{\pm}\right|\leq K\left|\lambda-\lambda_{\pm}\right|^{1/2},
\]
so for $c=\tilde{u}_{\pm}\pm\delta$ with $\delta>0$, we have 
\[
\left|\tilde{u}(\lambda)-c\right|\leq\left|\tilde{u}(\lambda)-\tilde{u}_{\pm}\right|+\delta\leq K\left|\lambda-\lambda_{\pm}\right|^{1/2}+\delta.
\]
Therefore, there exists a constant $L>0$ such that for all $\delta\in(0,1)$,
\[
F(c)=F(\tilde{u}_{\pm}\pm\delta)\geq \frac{g}{2} \inf_{I}H\,\int_{0}^{1}\frac{1}{K^{2}\left|\lambda-\lambda_{\pm}\right|+\delta^{2}}\mathrm{d}\lambda\geq L|\log\delta^2|,
\]
which proves 
\begin{align*}
    \lim_{c\to \tilde{u}_{-},c<\tilde{u}_{-}}F(c)=\infty \quad \textrm{ and }  \quad \lim_{c\to \tilde{u}_{+},c>\tilde{u}_{+}}F(c)=\infty.
\end{align*}
}
\item{If $\tilde{u}$ is $\frac{\ensuremath{1}}{4}$-Hölder continuous with
constant $K>0$, then for all $\lambda\in I$,
\[
\left|\tilde{u}(\lambda)-\tilde{u}_{\pm}\right|\leq K\left|\lambda-\lambda_{\pm}\right|^{1/4},
\]
therefore
\[
F(\tilde{u}_{\pm})\geq g \inf_I H \int_{0}^{1}\frac{1}{K^{2}\left|\lambda-\lambda_{\pm}\right|^{1/2}}\mathrm{d}\lambda=\frac{2g}{K^{2}}\inf_{I}H\left(\sqrt{1-\lambda_{\pm}}+\sqrt{\lambda_{\pm}}\right)\geq\frac{2g}{K^{2}}\inf_{I}H.
\]
Hence on the one hand if the constant $K$ satisfies $K\displaystyle<\sqrt{2g\inf_I H}$, then
\begin{align*}
    \lim_{c\to \tilde{u}_{-},c<\tilde{u}_{-}}F(c) > 1 \quad \textrm{ and }  \quad \lim_{c\to \tilde{u}_{+},c>\tilde{u}_{+}}F(c) > 1.
\end{align*}
}
\end{enumerate}

On the other hand, for all $\lambda\in \bar{I}$ one has $|\tilde{u}_{\pm}\pm\sqrt{gh}-\tilde{u}|\geq\sqrt{gh}$
so
\[
F(\tilde{u}_{\pm}\pm\sqrt{gh})\le1.
\]
Taking the derivative, one has
\[
F^{\prime}(c)=-\int_{0}^{1}\frac{2gH}{(c-\tilde{u})^{3}}\mathrm{d}\lambda,
\]
which shows that $F$ is strictly decreasing on $J_{+}$ and strictly
increasing on $J_{-}$. Therefore, there exists exactly one solution
of $F(c)=1$ in $J_{-}$ and exactly one in $J_{+}$.
\end{proof}

\begin{remark} \label{rem:optimal}
The previous result is almost optimal in two aspects. The first is that no smallness assumption is required in $C^{0,1/2}(\bar{I})$ whereas a smallness assumption is required in $C^{0,1/4}(\bar{I})$. In fact one can show that $C^{0,1/2}(\bar{I})$ is the critical space, as a (more complicated) smallness assumption is also required in $C^{0,\alpha}(\bar{I})$ for $\frac{1}{4}<\alpha<\frac{1}{2}$. The second aspect is that this result is also optimal despite the smallness assumption required in $C^{0,1/4}(\bar{I})$, in view of the following example.
Consider $H(\lambda)=1$ and $\tilde{u}(\lambda)=K\lambda^{1/4}$ for some $K>0$,
then $\tilde{\boldsymbol U}=(H,\tilde{u})$ is a stationary solution of system \eqref{eq:new_system_vector_1d}. An explicit computation shows that $\tilde{u}_{-}=0$,
$\tilde{u}_{+}=K$, and that $F$ is strictly decreasing on $J_{+}$ and strictly
increasing on $J_{-}$. Moreover, we have the following limits:
\begin{align*}
\lim_{c\to \tilde{u}_{-}} & F(c)=\frac{2g}{K^{2}}, & \lim_{c\to \tilde{u}_{+}}F(c) & =\infty.
\end{align*}
This proves that there are exactly two real eigenvalues when $K<\sqrt{2g}$
and exactly one (which is in $J_{+}$) when $K\geq\sqrt{2g}$, see \cref{fig:limitingcase}.
\end{remark}
\begin{figure}
    \centering
    \includegraphics{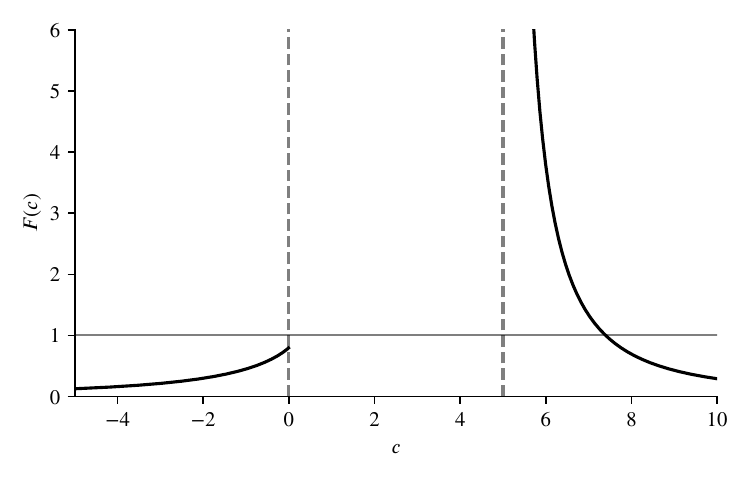}
    \caption{The limiting case for $\tilde{u}(\lambda)=K\lambda^{1/4}$, $H(\lambda)=1$, and $K\geq\sqrt{2g}$ for $g=10$.}
    \label{fig:limitingcase}
\end{figure}

Finally, under additional hypotheses, we can ensure that the spectrum is fully real.
\begin{proposition} \label{prop:u_convexe}
Let $H\in C^{1}(\bar{I})$ with $H>0$ and $\tilde{u}\in C^2(\bar{I})$ be strictly monotonic in $\lambda$. If the vorticity $\tilde{\omega}=\partial_{\lambda}\tilde{u}/H$ defined in \cref{prop:vorticity_new} satisfies either:
\begin{enumerate}
\item $\partial_{\lambda}\tilde{\omega}\neq0$ for all $\lambda\in\bar{I}$,
\item or there exists $\lambda_{c}\in I$ such that $\partial_{\lambda}\tilde{\omega}(\lambda_{c})=0$ and
 $\partial_{\lambda}\tilde{\omega}\left(\tilde{u}-\tilde{u}(\lambda_{c})\right)>0$ for all $\lambda\neq\lambda_{c}$,
\end{enumerate}

then
\begin{align*}
    \sigma_{p}(A_0) &= \sigma_{d}(A_0) = \left\{ c_{-},c_{+}\right\}, &
    \sigma_{c}(A_0) &= \sigma_{e}(A_0) = \tilde{u}(\bar{I}), &
    \sigma_{r}(A_0) &= \emptyset,
\end{align*}
where $c_{\pm}$ are defined in \cref{prop:limit_case}.
\end{proposition}

\begin{remark}
Note that in the original variables $\partial_{\lambda}\tilde{\omega}=H\partial_{zz}u$ so that
the first

condition $\partial_{\lambda}\tilde{\omega}\neq 0$ is equivalent to the condition $\partial_{zz}u \neq 0$ in the original domain $\Omega_t$.

Moreover, the second condition $\partial_{\lambda}\tilde{\omega}(\lambda_{c})=0$ and
$\partial_{\lambda}\tilde{\omega}\left(\tilde{u}-\tilde{u}(\lambda_{c})\right)>0$ for $\lambda\neq\lambda_{c}$ is equivalent to $\partial_{zz}u(z_c)=0$ and $\partial_{zz}u\left(u-u(z_c)\right)>0$ for all $z_c\neq z$.

We note that the first condition is precisely the condition proposed in~\cite[Lemma 2.2)]{chesnokov2017} to study the stability of the flow using the Sokhotski–Plemelj formulae.

Finally, it is interesting to remark that these two conditions are precisely the classical Rayleigh and Fjørtoft conditions on the linear instability of shear flows without free-surface (see for example~\cite[§8.2]{Drazin-stability2002}).
\end{remark}

\begin{proof}
Since $\tilde u$ is strictly monotonic, we have that $\meas\bigl(\tilde{u}^{-1}(c)\bigr)=0$ for $c\in\tilde{u}(\bar{I})$, hence $\sigma_{p}(A_0)=\sigma_{d}(A_0)$ and $\sigma_{c}(A_0) = \sigma_{e}(A_0) = \tilde{u}(I)$.
Therefore, it remains to study the function $F$ defined by \eqref{eq:F_c} on $\mathbb{C}\setminus\tilde{u}(\bar{I})$. In view of \cref{prop:spectrum_localization,prop:limit_case} it suffices to prove that $F(c)\neq1$ for all $c\in\mathbb{C}$ such that $\Im c\neq0$ and $\tilde{u}_{-}\leq \Re c\leq\tilde{u}_{+}$. In the following we assume that $c$ satisfies these conditions.

Without loss of generality (by making the change $\tilde{u}\mapsto-\tilde{u}$), we can assume that $\partial_\lambda \tilde u>0$, and therefore, we have

\[
F(c)=\int_{0}^{1}\frac{gH}{(c-\tilde{u})^{2}}\mathrm{d}\lambda=\int_{0}^{1}\frac{1}{\tilde{\omega}}\partial_{\lambda}\Bigl(\frac{g}{c-\tilde{u}}\Bigr)\mathrm{d}\lambda=\left[\frac{g}{\tilde{\omega}(c-\tilde{u})}\right]_{0}^{1}+\int_{0}^{1}\frac{\partial_{\lambda}\tilde{\omega}}{\tilde{\omega}^{2}}\frac{g}{c-\tilde{u}}\mathrm{d}\lambda,
\]
from which we get
\begin{align*}
\Re F(c) & =g\left[\frac{\Re c-\tilde{u}}{\tilde{\omega}|c-\tilde{u}|^{2}}\right]_{0}^{1}+g\int_{0}^{1}\frac{\partial_{\lambda}\tilde{\omega}}{\tilde{\omega}^{2}}\frac{\Re c-\tilde{u}}{|c-\tilde{u}|^{2}}\mathrm{d}\lambda,\\
\Im F(c) & =-\Im c\left[\frac{g}{\tilde{\omega}|c-\tilde{u}|^{2}}\right]_{0}^{1}-\Im c\int_{0}^{1}\frac{\partial_{\lambda}\tilde{\omega}}{\tilde{\omega}^{2}}\frac{g}{|c-\tilde{u}|^{2}}\mathrm{d}\lambda.
\end{align*}

We now distinguish the two cases:
\begin{enumerate}
\item Since $\partial_{\lambda}\tilde{\omega}\neq0$, by continuity $\partial_{\lambda}\tilde{\omega}$
does not change sign on $\bar{I}$ so by using the mean value theorem
on $\Re F(c)$, there exists $\lambda_{c}\in\bar{I}$ such that
\[
\Re F(c)=g\left[\frac{\Re c-\tilde{u}}{\tilde{\omega}|c-\tilde{u}|^{2}}\right]_{0}^{1}+\left(\Re c-\tilde{u}(\lambda_{c})\right)\int_{0}^{1}\frac{\partial_{\lambda}\tilde{\omega}}{\tilde{\omega}^{2}}\frac{g}{|c-\tilde{u}|^{2}}\mathrm{d}\lambda.
\]
\item Since $\partial_{\lambda}\tilde{\omega}\,\tilde{u}\geq\partial_{\lambda}\tilde{\omega}\,\tilde{u}(\lambda_{c})$,
we have
\[
\Re F(c)\leq g\left[\frac{\Re c-\tilde{u}}{\tilde{\omega}|c-\tilde{u}|^{2}}\right]_{0}^{1}+g\int_{0}^{1}\frac{\partial_{\lambda}\tilde{\omega}}{\tilde{\omega}^{2}}\frac{\Re c-\tilde{u}(\lambda_{c})}{|c-\tilde{u}|^{2}}\mathrm{d}\lambda.
\]
\end{enumerate}

Therefore, in either case, we have
\[
\Re F(c)\leq g\left[\frac{\Re c-\tilde{u}}{\tilde{\omega}|c-\tilde{u}|^{2}}\right]_{0}^{1}+\left(\Re c-\tilde{u}(\lambda_{c})\right)\int_{0}^{1}\frac{\partial_{\lambda}\tilde{\omega}}{\tilde{\omega}^{2}}\frac{g}{|c-\tilde{u}|^{2}}\mathrm{d}\lambda,
\]
for some $\lambda_{c}\in\bar{I}$. Now the aim is to combine $\Re F(c)$
and $\Im F(c)$ to remove the integral. Since $\Im c\neq0$, one has
\begin{align*}
\Re F(c)+\frac{\Re c-\tilde{u}(\lambda_{c})}{\Im c}\Im F(c) & \leq g\left[\frac{\Re c-\tilde{u}}{\tilde{\omega}|c-\tilde{u}|^{2}}\right]_{0}^{1}-\left(\Re c-\tilde{u}(\lambda_{c})\right)\left[\frac{g}{\tilde{\omega}|c-\tilde{u}|^{2}}\right]_{0}^{1}\\
 & =\left[\frac{g}{\tilde{\omega}|c-\tilde{u}|^{2}}\left(\tilde{u}(\lambda_{c})-\tilde{u}\right)\right]_{0}^{1}.
\end{align*}

Since $\tilde{u}(0)\leq\tilde{u}(\lambda_{c})\leq\tilde{u}(1)$ and $\tilde{\omega}>0$,
we deduce that 
\[
    \Re F(c)+\frac{\Re c-\tilde{u}(\lambda_{c})}{\Im c}\Im F(c)\leq0,
\]
so
\begin{equation}\label{eq:convexe_bound_Fc}
    \Re F(c)\leq\frac{|\Re c-\tilde{u}(\lambda_{c})|}{|\Im c|}|\Im F(c)|
\end{equation}
which is incompatible with having a solution of $F(c)=1$.
\end{proof}

None of the two assumptions of \cref{prop:u_convexe} can be omitted regardless of the sign of the derivative $\partial_{\lambda}\tilde{u}$ as enlightened by the following counter-example showing the existence of complex eigenvalues, hence possible instability.
\begin{proposition} \label{prop:u_convexe_counterexample}
For $a\neq0$ and $b>0$ such that
\begin{align*}
    b\tanh b &> 1, &
    |a| &< \sqrt{1-(b\tanh b)^{-1}},
\end{align*}
the spectrum of $A_0$ for
\begin{align*}
    H(\lambda) &= g^{-1}, &
    \tilde{u}(\lambda) &= a \tanh(b(2\lambda-1))\,
\end{align*}
has at least two complex conjugate eigenvalues with $\Re c=0$ and $\Im c\neq0$.
\end{proposition}

\begin{proof}
The aim is to study the function $F(c)$ defined by \eqref{eq:F_c} for $c\in i\R^*$, \emph{i.e.} $\Re c=0$ and $\Im c\neq0$. For simplicity, we write $\nu=\Im c$, so
\begin{align*}
    \Re F(i\nu) &= \int_{0}^{1}\frac{\tilde{u}^2 - \nu^2}{(\tilde{u}^2+\nu^2)^2}\mathrm{d}\lambda, &
    \Im F(i\nu) &= 2\int_{0}^{1}\frac{\tilde{u}\nu}{(\tilde{u}^2+\nu^2)^2}\mathrm{d}\lambda,
\end{align*}
The function $\tilde{u}$ being odd with respect to $\lambda=\frac{1}{2}$, we deduce directly that $\Im F(i\nu)=0$.
We can compute $\Re F(i\nu)$ explicitly and in particular we can show that:
\begin{align*}
    \lim_{\nu\to0} \Re F(i\nu) &= \frac{1-(b\tanh b)^{-1}}{a^2}, &
    \lim_{\nu\to\pm\infty} \Re F(i\nu) &= 0.
\end{align*}
Therefore, under the above hypotheses, $ \lim_{\nu\to0} \Re F(i\nu)>1$ and  by continuity, there exists some $\nu\in(0,\infty)$ such that $\Re F(i\nu)=1$.
\end{proof}

\subsection{Generalized Riemann invariants}\label{subsec:riemann}

In this section, we attempt to define a generalized notion of Riemann invariant associated to the eigenvalues $c$, by analogy to the classical theory of hyperbolic systems.
For simplicity only the case $d=1$ is considered.
Given a solution $\boldsymbol{\tilde{U}}$ to \eqref{eq:new_system_vector_1d}, which we assume regular enough, an associated eigenvalue $c\in\sigma_p(A_0)$ depends implicitly on $t$ and $x$ (through $\boldsymbol{\tilde{U}}$ by \eqref{eq:condition-eigenvalues}), but not on $\lambda$.
More precisely, the function $c(t,x)$ is a solution of $F(t,x,c(t,x))=1$, where:
\[
    F(t,x,c) = \int_{0}^{1}\frac{gH(t,x,\lambda)}{(c-\tilde{u}(t,x,\lambda))^{2}}\mathrm{d}\lambda.
\]
Since $\partial_c F(t,x,c)$ is strictly negative on $(0,T)\times\R\times J_{+}$ and strictly positive on $(0,T)\times\R\times J_{-}$, the implicit function theorem ensures the regularity of the function $c(t,x)$ from the regularity of $H$ and $\tilde{u}$.
We shall call a generalized Riemann invariant associated to $c(t,x)$, any function $R(t,x)$ satisfying:
\[
    \frac{\partial R}{\partial t}+c\,\frac{\partial R}{\partial x}=0.
\]
This new definition, introduced in analogy with the Saint-Venant system,  allows us to view the generalized Riemann invariants as quantities transported by the velocity $c(t,x)$ for all $\lambda\in I$, see \cite{teshukov2004}.

\begin{proposition}\label{prop:riemann_inv}
Let $c=c(t,x)\in\sigma_p(A_0)$  be a solution of \eqref{eq:condition-eigenvalues}, then the relation
\begin{equation}
\frac{\partial}{\partial t} \left( c - g\int_0^1 \frac{H}{\tilde u - c}\mathrm{d}\lambda\right) + c\frac{\partial}{\partial x} \left( c - g\int_0^1 \frac{H}{\tilde u - c}\mathrm{d}\lambda\right)= 0,
\label{eq:r_invariants_sanstopo}
\end{equation}
holds and the quantity $\tilde r = \tilde r(t,x)$ defined by 
\[
    \tilde r = c - g\int_0^1 \frac{H}{\tilde u - c}\mathrm{d}\lambda,
\]
can be seen as a generalized Riemann invariant associated to the eigenvalue $c=c(t,x)$.
\end{proposition}

\begin{proof}
Let us consider $c=c(t,x)\in\sigma_p(A_0)$. Computing the derivatives in \eqref{eq:r_invariants_sanstopo} and replacing the time derivatives of $H$ and $\tilde{u}$ using \eqref{eq:new_system} leads to:
\[
\partial_{t}\tilde{r}+c\partial_{x}\tilde{r}=\left(\partial_{t}c+c\partial_{x}c+g\int_{0}^{1}\partial_{x}H\mathrm{d}\lambda\right)\left[1-\int_{0}^{1}\frac{gH}{(c-\tilde{u})^{2}}\mathrm{d}\lambda\right].
\]
The result follows using the fact that $c$ satisfies~\eqref{eq:condition-eigenvalues}.

As for the classical hyperbolic Saint-Venant system, the result can also be obtained multiplying the system~\eqref{eq:new_system_vector_1d} on the left by $\bvphi$, where $\bvphi$ is a vector orthogonal to an eigenvector $\bphi$ of $A_0$, see \eqref{eq:redsol_phi_sanstopo1}-\eqref{eq:redsol_phi_sanstopo2}. Then after simple computations, an integration in $\lambda$ on $I=(0,1)$  gives the result.
\end{proof}

Moreover, in view of \cref{prop:vorticity_new}, the vorticity $\tilde{\omega}=\partial_{\lambda}\tilde{u}/H$ satisfies the transport equation:
\[
    \frac{\partial\tilde{\omega}}{\partial t}+\tilde{u}\frac{\partial\tilde{\omega}}{\partial x}=0,
\]
therefore, assuming that $\tilde{\omega}$ is smooth enough, for a fixed value of $\lambda\in I$ one can consider $\tilde{\omega}_{\lambda}(t,x) = \tilde{\omega}(t,x,\lambda)$ as a generalized Riemann invariant associated to the velocity $\tilde{u}_{\lambda}(t,x)=\tilde{u}(t,x,\lambda)$ in $\sigma_{e}(A_{0})$.

Together with the the existence of two real eigenvalues ensured by \cref{prop:limit_case}, we deduce the following result.
\begin{corollary}
If $\boldsymbol{\tilde{U}}$ is an enough regular solution to \eqref{eq:new_system_vector_1d}, let $c_\pm=c_\pm(t,x)$ denote the two real eigenvalues given by \cref{prop:limit_case}. We have the following list of equations for any $\lambda\in I$:
\begin{align*}
\frac{\partial r_{\pm}}{\partial t} + c_{\pm}\frac{\partial r_{\pm}}{\partial x} &= 0,&
\frac{\partial \tilde{\omega}_{\lambda}}{\partial t} + \tilde{u}_{\lambda} \frac{\partial \tilde{\omega}_{\lambda}}{\partial x} &= 0,
\end{align*}
where the Riemann invariants are defined by:
\begin{align*}
r_{\pm} &= c_{\pm} - g\int_0^1 \frac{H}{\tilde u - c_{\pm}}\mathrm{d}\lambda, &
\tilde{\omega}_{\lambda}(t,x) &= \tilde{\omega}(t,x,\lambda) = \frac{\partial_\lambda\tilde{u}}{H}, &
\tilde{u}_{\lambda}(t,x) &= \tilde{u}(t,x,\lambda).
\end{align*}
\end{corollary}

The previous result is an attempt to somehow diagonalize \eqref{eq:new_system_vector_1d} and is consistent with the decomposition of the spectrum given by \cref{cor:spectrum}: the two real eigenvalues $c_{\pm}$ in $\sigma_d(A_0)$ lead to two equations, and the continuous spectrum $\tilde{u}(t,x,I)$ to an infinite number of equations parameterized by $\lambda$.
However, in the case where complex eigenvalues exist in $\sigma_d(A_0)$, it seems quite complicated to understand and justify how the previous list of Riemann invariant equations are equivalent to the original system \eqref{eq:new_system_vector_1d}.

\begin{remark}
We note that Besse diagonalized in \cite{Besse2011} a system very similar to \eqref{eq:new_system_vector_1d}, mainly when $A_0$ is replaced by
\[
    A_{0}(\tilde{\boldsymbol{U}})=\begin{pmatrix}
        \tilde{u} & H\\
        \frac{H}{4}+g\int_{0}^{1}\cdot\,\mathrm{d}\lambda & \tilde{u}
    \end{pmatrix}.
\]
It would be very interesting to see if the techniques he developed could also be applied in our case.
\end{remark}

\begin{remark}
Riemann invariants are often used to build particular solutions of the considered system as in the construction of simple waves (see for instance \cite{teshukov1997}). This has not been investigated in this paper but this could help to understand the behavior of the system~\eqref{eq:new_system_vector_1d} in valuable situations.
\end{remark}

\subsection{Case with variable topography}\label{subsec:avectopo}
In the classical Saint-Venant system, one way to deal with non trivial topography allowing to obtain an hyperbolic system is to add $z_b$ as a variable together with the equation $\partial_{t}z_b=0$, see for example \cite{godlewski-book, gosse1996, roux1996}.
For $d=2$ and by following the same idea, we can rewrite \eqref{eq:new_system_vector} coupled with topography $z_b(t,\bx)$ satisfying $\partial_{t}z_b=0$:
\begin{equation}\label{eq:new_system_vector_topo}
    \frac{\partial\boldsymbol{\hat U}}{\partial t}+B_{1}(\boldsymbol{\hat U})\frac{\partial\boldsymbol{\hat U}}{\partial x}+B_{2}(\boldsymbol{\hat U})\frac{\partial\boldsymbol{\hat U}}{\partial y}=0,
\end{equation}
where $\boldsymbol{\hat U}=(H,\tilde{\boldsymbol{u}},z_b)^{T}$ and
\begin{align*}
    B_{1}(\boldsymbol{\hat U}) & =\begin{pmatrix}
        \tilde{u} & H & 0 & 0\\
        g\int_{0}^{1}\cdot \,\mathrm{d}\lambda & \tilde{u} & 0 & g\\
        0 & 0 & \tilde{u} & 0\\
        0 & 0 & 0 & 0
    \end{pmatrix}, &
    B_{2}(\boldsymbol{\hat U}) & =\begin{pmatrix}
        \tilde{v} & 0 & H & 0\\
        0 & \tilde{v} & 0 & 0\\
        g\int_{0}^{1}\cdot \,\mathrm{d}\lambda & 0 & \tilde{v} & g\\
        0 & 0 & 0 & 0
    \end{pmatrix},
\end{align*}
are non-local operators acting on functions in the variable $\lambda$. Hence, we will again write $\tilde{\boldsymbol{u}}(\lambda)$ for the sake of simplicity. As in the case of constant topography presented above, system \eqref{eq:new_system_vector_topo} is invariant by rotation along the vertical axis and by a similar argument, this motivates the analysis of the $3\times 3$ matrix,
\[
    B_0(\bar U) =\begin{pmatrix}
        \tilde{u} & H & 0\\
        g\int_{0}^{1}\cdot\,\mathrm{d}\lambda & \tilde{u} & g\\
        0 & 0 & 0
    \end{pmatrix}.
\]
Provided that $\tilde{u}\in L^{\infty}(I)$ and $H\in L^{\infty}(I)$, $B_0(\bar U):L^{2}(I)^{3}\rightarrow L^{2}(I)^{3}$ is a bounded operator. We shall consider the following system corresponding to the case $d=1$:
\[
    \frac{\partial \bar U}{\partial t} + B_0(\bar U) \frac{\partial \bar U}{\partial x} = 0,
    \quad \text{where} \quad
    \bar U =\begin{pmatrix}
        H\\
        \tilde{u}\\
        z_b
    \end{pmatrix}.
\]

In this case the spectrum decomposition is not changed much, except that zero is always in the point spectrum, similarly to the fact that zero is an eigenvalue for the Saint-Venant system with topography.

\begin{theorem}\label{thm:spectrum-topo}
 If $\tilde{u}\in C^{0,1/4}(\bar{I})$, $H\in C^{0}(\bar{I})$ and $H>0$ on $\bar{I}$, then the
 spectrum of $B_0(\bar U):L^{2}(I)^{3}\to L^{2}(I)^{3}$ is characterized by
\begin{align*}
 \sigma_{p}(B_0) & = \{0\} \cup \left\{ c\in\mathbb{C}\setminus\tilde{u}(\bar{I}) : \int_{0}^{1}\frac{gH}{(c-\tilde{u})^{2}}\mathrm{d}\lambda=1\right\} \cup\left\{ c\in\tilde{u}(\bar{I})\setminus \{0\}\,:\:\meas\bigl(\tilde{u}^{-1}(c)\bigr)>0\right\},\\
\sigma_{c}(B_0) & =\left\{ c\in\tilde{u}(\bar{I})\setminus \{0\}\,:\:\meas\bigl(\tilde{u}^{-1}(c)\bigr)=0\right\},\\
\sigma_{r}(B_0) & =\emptyset,
\end{align*}
where the definitions of the point spectrum $\sigma_{p}(B_{0})$,
continuous spectrum $\sigma_{c}(B_{0})$, and residual spectrum $\sigma_{r}(B_{0})$ are given in \cref{sec:appendix}.
\end{theorem}

The proof is very similar to the proof of \cref{thm:spectrum} and is omitted. The only change is that we have to distinguish the cases $c=0$ and $c\neq0$. In addition \cref{prop:spectrum_localization,prop:limit_case,prop:u_convexe} also hold with completely similar proofs.

In the shallow water regime, the spectrum reduces as explained in \cref{rem:spectrum_saint_venant} however with zero as an addition eigenvalue, like for the classical Saint-Venant system with topography.

The Riemann invariants  corresponding to the eigenvalues $c\in\sigma_p(B_0)$ are slightly modified to take into account the topography.
\begin{proposition}\label{prop:riemann_inv_avec_topo}
Let $c=c(t,x)\in\sigma_p(B_0)$ be a solution of \eqref{eq:condition-eigenvalues}, then the relation
\[
    \frac{\partial}{\partial t} \left( c - g\int_0^1 \frac{H}{\tilde u - c}\mathrm{d}\lambda - g z_b\right) + c\frac{\partial}{\partial x} \left( c - g\int_0^1 \frac{H}{\tilde u - c}\mathrm{d}\lambda - g z_b\right)= 0,
\]
holds and the quantity $\bar r = \bar r(t,x)$ with
\[
    \bar r = c - g\int_0^1 \frac{H}{\tilde u - c}\mathrm{d}\lambda - g z_b,
\]
can be seen as a generalized Riemann invariant associated to the eigenvalue $c=c(t,x)$.
\end{proposition}

\begin{remark}
For $c=0$, the associated Riemann invariant is trivially $z_b$ as $\partial_{t}z_b=0$. However in case $c=0$ is solution of \eqref{eq:condition-eigenvalues}, another Riemann invariant given by the previous proposition might exist.
\end{remark}

\section{A multilayer approach} \label{sec:multi_layer}

In order to better understand the behavior of the eigenvalues in the interval $\tilde{u}(\bar{I})$, and thinking about a possible numerical approximation, we will use a multilayer discretization of the vertical domain. This approach was introduced by Audusse et al. (see~\cite{JSM_M2AN}) in order to access the vertical variation of the horizontal velocity profile of the Euler system. Following this approach we consider a piecewise constant approximation of the velocity $\tilde{u}$ in the variable $\lambda\in \bar{I}$ and we show in \cref{prop:discrete_eigen} that for a monotone velocity profile and for a large number of horizontal layers we recover only two real eigenvalues $c_{\pm}$ corresponding to system \eqref{eq:new_system_vector_1d} under the assumptions of \cref{prop:limit_case}.
Note that a piecewise linear approximation (discontinuous and continuous) was introduced in~\cite{chesnokov2017}.

The vertical coordinate plays a particular role in geophysical flow models. For the numerical approximation of the Navier--Stokes or Euler equations with free surface, multilayer models have been proposed~\cite{audusse2005,JSM_JCP,BDGSM,chesnokov2017,fernandeznieto} consisting in a piecewise constant approximation of the horizontal velocity field along the vertical axis. Efficient and robust numerical schemes endowed with strong stability properties (well-balancing, discrete maximum principle, discrete entropy inequality) can be obtained to approximate these multilayer models~\cite{art_3d}. In this section we propose and study a multilayer version of the model~\eqref{eq:new_system}. In contrast to other multilayer schemes for other problems, in the present case, the piecewise constant approximation is exact, \emph{i.e.}, the truncation error is zero.
The objective is to examine the eigenvalues of this multilayer version of the model~\eqref{eq:new_system}.

\medskip
In this section, we only consider $d=1$.
As described by \cref{fig:multi_lambda}, we consider a discretization of the fluid domain $\tilde\Omega=\R\times I$ into $N$ different layers. The layer with index $\alpha$ contains the points of coordinates $(x,\lambda)$ with $\lambda\in {L}_{\alpha}=[\lambda_{\alpha-1/2},\lambda_{\alpha+1/2}]$ with
\[
    0=\lambda_{1/2}<\lambda_{3/2}<\cdots<\lambda_{N+1/2}=1.
\]
For $\alpha =1,\dots,N$, the width of the layer $\alpha$ is given by
\[
    \gamma_\alpha = \lambda_{\alpha+1/2}-\lambda_{\alpha-1/2},
\]
with $\sum_{\beta=1}^N \gamma_\beta = 1$. Note that each $\lambda_{\alpha+1/2}$ is a constant corresponding to the interface between the two layers  $\alpha$ and $\alpha+1$ with 
\[
    \lambda_{\alpha+1/2}=\sum_{\beta=1}^\alpha \gamma_{\beta}.
\]

For a given $X: \tilde\Omega\rightarrow\R$, we consider its $\mathbb{P}_0-$approximation in $\lambda$ having the form:
\[
    X^N(t,x,\lambda) := \sum_{\alpha=1}^N\1_{\lambda \in L_\alpha}(\lambda) X_\alpha(t,x),
\]
where $X_\alpha(t,x)$ defined by
\begin{align}\label{eq:p0_app}
X_\alpha(t,x)=\displaystyle\frac{1}{\gamma_\alpha}\int_{\lambda_{\alpha-1/2}}^{\lambda_{\alpha+1/2}} X(t,x,\lambda)\,\mathrm{d}\lambda,
\end{align}
is the average of $X(t,x,\lambda)$ over the layer $L_\alpha$.
We will denote by $\boldsymbol{X}_N$ the vector $(X_1,\dots,X_N)^T$. Then the following proposition holds.

\begin{figure}
\centering
\begin{tikzpicture}[xscale=1.5,yscale=1]
    \draw[thick,->] (0,0) -- (5,0) node[right] {$x$};
    \draw[thick, -] (0,3) -- (5,3) node[right] {$\lambda_{N+1/2}$};
    \draw[thin, dashed] (0,2.5) -- (5,2.5);
    \draw[thin, dashed] (0,2) -- (5,2) node[right] {$\lambda_{\alpha+1/2}$};
    \fill[nearly transparent](0,1.5) rectangle (5,2);
    \draw (2,1.75) node[right]{$\gamma_{\alpha}=        \lambda_{\alpha+1/2}-\lambda_{\alpha-1/2}$};
    \draw[thick, <->] (1.8,1.5) -- (1.8,2);
    \draw[thin, dashed] (0,1.5) -- (5,1.5) node[right] {$\lambda_{\alpha-1/2}$};
    \draw[thin, dashed] (0,1) -- (5,1);
    \draw[thin, dashed](0,.5) -- (5,.5);
    \draw[thick,->] (0,0) -- (0,4) node[right] {$\lambda$};
    \draw (0,0) node[left] {$0$};
    \draw (0,3) node[left] {$1$};
\end{tikzpicture}
\caption{Multilayer discretization of $\tilde\Omega$}
\label{fig:multi_lambda}
\end{figure}

\begin{proposition}
Let $(H,\tilde u)$ be a solution of the system~\eqref{eq:new_system} completed with initial conditions \eqref{eq:new_system_init}. For $\alpha=1,..., N$, the multilayer formulation given by
\begin{equation} \label{eq:sv_change}
    \left\{
    \begin{aligned}
        & \frac{\partial H_\alpha}{\partial t} + \partial_x (H_\alpha\tilde u_\alpha) = 0,\\
        & \frac{\partial \tilde u_\alpha}{\partial t} + \tilde u_\alpha \partial_x \tilde u_\alpha + g \partial_{x} \sum_{\beta=1}^N \gamma_\beta H_\beta = -g\partial_x z_b,\\
    \end{aligned}
    \right.
\end{equation}
with initial conditions:
\begin{equation} \label{eq:init_msv}
    \left\{
    \begin{aligned}
       H_\alpha(0,x) &= \frac{1}{\gamma_\alpha}\int_{\lambda_{\alpha-1/2}}^{\lambda_{\alpha +1/2}} H(0,x,\lambda)\,\mathrm{d}\lambda,\\
       \tilde{u}_\alpha(0,x) &= \frac{1}{\gamma_\alpha}\int_{\lambda_{\alpha-1/2}}^{\lambda_{\alpha +1/2}} \tilde u(0,x,\lambda)\,\mathrm{d}\lambda,
    \end{aligned}
    \right.
\end{equation}
is the $\mathbb{P}_0$ Galerkin approximation of \eqref{eq:new_system}.
\label{prop:model_mc}
\end{proposition}

\begin{proof}
The one-dimensional form of equations \eqref{eq:new_system} reads 
\begin{align*}
& \frac{\partial H}{\partial t} + \partial_x (H\tilde u)
 = 0,
& \frac{\partial \tilde u}{\partial t } + \tilde u \partial_x \tilde u
      + g\partial_x \int_0^1 H \mathrm{d}\lambda = -g \partial_x z_b.
 \end{align*}
Multiplying these equations by $\1_\alpha = \1_{\lambda\in L_{\alpha}}$ and integrating over $\lambda\in I$, we get
\begin{align*}
& \frac{\partial}{\partial t} \int_0^1 H\1_\alpha \mathrm{d}\lambda + \partial_x \int_0^1 H\tilde u\1_\alpha \mathrm{d}\lambda
 = 0,\\
& \frac{\partial }{\partial t } \int_0^1 \tilde u\1_\alpha \mathrm{d}\lambda + \frac{1}{2}\partial_x \int_0^1 \tilde u^2 \1_\alpha \mathrm{d}\lambda 
      + g \gamma_\alpha\partial_x \int_0^1 H \mathrm{d}\lambda = -g \gamma_\alpha\partial_x  z_b.
\end{align*}
Replacing $H$ and $\tilde u$ by their approximations defined by~\eqref{eq:p0_app}, we recover the multilayer formulation under the form \eqref{eq:sv_change} for $\alpha=1,\ldots,N$.
\end{proof}

The following corollary emphasizes the interest of the multilayer model for a class of piecewise constant functions.
\begin{corollary} \label{cor:multi_layer_exact}
Let us consider $(H,\tilde u)$ defined by
\begin{align*}
    H &= \sum_{\alpha=1}^N\1_{\lambda \in L_\alpha}(\lambda) H_\alpha(t,x), &
    \tilde u &= \sum_{\alpha=1}^N\1_{\lambda \in L_\alpha}(\lambda) \tilde u_\alpha(t,x).
\end{align*}
Then $(H,\tilde u)$ are solution of the system~\eqref{eq:new_system} if and only if $(\boldsymbol{H}_N,\tilde{\boldsymbol{U}}_N)$ defined by $\boldsymbol{H}_N=(H_1,\ldots,H_N)^T$ and $\tilde{\boldsymbol{U}}_N=(\tilde u_1,\ldots,\tilde u_N)^T$ are solution of the system~\eqref{eq:sv_change} with initial data \eqref{eq:init_msv}.
\end{corollary}

\begin{proof}
The corollary is deduced directly from the previous proof of \cref{prop:model_mc}, the approximation commuting with the nonlinearities and integration.
\end{proof}

We remark that system \eqref{eq:sv_change} can be viewed as the Saint-Venant equations in each layer with some coupling through the summation term. In particular the interaction is between all the layers and not only adjacent ones; in fact there is no exchange term between the layers and the coupling is exactly the same for all the layers.

Interestingly, in the multilayer case, we can go back to the original
domain $\Omega_{t}$ as explained in \cref{thm:reciproque} explicitly.
\begin{proposition} \label{prop:multi_layer_back}
    The solution $(H,\tilde{u})$ given by \cref{cor:multi_layer_exact}
    in $\tilde{\Omega}$ corresponds to the solution $(\eta,u)$ in $\Omega_t$ defined by
    \begin{align*}
        \eta(t,x) & =z_{b}(x)+\sum_{\alpha=1}^{N}\gamma_{\alpha}H_{\alpha}(t,x),&
        u(t,x,z) & =\sum_{\alpha=1}^{N}\1_{z\in M_{\alpha}(t,x)}(z)\tilde{u}_{n}(t,x),
    \end{align*}
    where $M_{\alpha}(t,x)$ denotes the following moving layer
    \[
        M_{\alpha}(t,x)=\biggl[z_{b}(x)+\sum_{\beta=1}^{\alpha-1}\gamma_{\beta}H_{\beta}(t,x),z_{b}(x)+\sum_{\beta=1}^{\alpha}\gamma_{\beta}H_{\beta}(t,x)\biggr].
    \]
    Note that the expression of the vertical velocity $w$ is more complicated,
    but is not needed in view of \cref{rem:w_from_divu}.
\end{proposition}
\begin{proof}
Let $(H_{\alpha},\tilde{u}_{\alpha})$ be solution of \cref{eq:sv_change}
and define
\begin{align*}
    H & =\sum_{\alpha=1}^{N}\1_{\lambda\in L_{\alpha}}(\lambda)H_{\alpha}(t,x), & \tilde{u} & =\sum_{\alpha=1}^{N}\1_{\lambda\in L_{\alpha}}(\lambda)\tilde{u}_{\alpha}(t,x).
\end{align*}
The first step is to determine $\phi$. Since $H=\partial_{\lambda}H$,
then for any $\lambda\in L_{\alpha}$,
\begin{align*}
    \phi(t,x,\lambda) & =z_{b}(x)+\int_{0}^{\lambda}H(t,x,\lambda^{\prime})\mathrm{d}\lambda^{\prime} =z_{b}(x)+\sum_{\beta=1}^{\alpha-1}\gamma_{\beta}H_{\beta}(t,x)+(\lambda-\lambda_{\alpha-1/2})H_{\alpha}(t,x),
\end{align*}
so in particular in view of \cref{thm:reciproque}, the free surface
is given by
\[
    \eta(t,x)=\phi(t,x,1)=z_{b}(x)+\sum_{\alpha=1}^{N}\gamma_{\alpha}H_{\alpha}(t,x).
\]
This leads to the following spatial discretization of the interval
$[z_{b}(x),\eta(t,x)]$ of $\Omega_{t}$ in $N$ moving layers $M_{\alpha}(t,x)$:
\[
    M_{\alpha}(t,x)=[z_{b}(x)+\sum_{\beta=1}^{\alpha-1}\gamma_{\beta}H_{\beta}(t,x),z_{b}(x)+\sum_{\beta=1}^{\alpha}\gamma_{\beta}H_{\beta}(t,x)].
\]
Since the interval $L_{\alpha}$ is mapped by $\phi(t,x)$ into $M_{\alpha}$,
we can easily inverse $\phi$ and obtain that for any $z\in M_{\alpha}(t,x)$,
\[
    \phi^{-1}(t,x,z)=\gamma_{\alpha-1/2}+\frac{1}{H_{\alpha}}\left[z-z_{b}(x)-\sum_{\beta=1}^{\alpha-1}\gamma_{\beta}H_{\beta}(t,x)\right].
\]
Therefore for any $z\in M_{\alpha}(t,x)$,
\[
    u(t,x,z)=\tilde{u}(t,x,\phi^{-1}(t,x,z))=\1_{z\in M_{\alpha}(t,x)}(z)\tilde{u}_{\alpha}(t,x),
\]
and the claimed expression for horizontal velocity $u$ is proven.
\end{proof}

\begin{remark}

We remark that system~\eqref{eq:sv_change} can also be written under the form
\begin{equation} \label{eq:sv_change_cons}
    \left\{
    \begin{aligned}
        & \frac{\partial H_\alpha}{\partial t} + \partial_x (H_\alpha\tilde u_\alpha) = 0,\\
        & \frac{\partial (H_\alpha\tilde u_\alpha)}{\partial t} + \partial_x \bigl(H_\alpha\tilde u_\alpha^2\bigr) + gH_\alpha\partial_x \biggl(\sum_{\beta=1}^N \gamma_\beta H_\beta\biggr) = -g H_\alpha\partial_x z_b,
    \end{aligned}
    \right.
\end{equation}
It is interesting to note that \eqref{eq:sv_change_cons} (with constant topography) corresponds exactly to the multilayer model (with zero viscosity) proposed by Audusse in \cite{audusse2005}. However, the derivation is pretty different.
In Audusse, the layers are assumed to be made of different immiscible fluids without exchange between them, and the model is a formal asymptotic expansion of the hydrostatic model.
However, in view of \cref{cor:multi_layer_exact}, equation~\eqref{eq:sv_change} provides a special class of solutions to system~\eqref{eq:new_system} without any approximation (except on the initial data).

In terms of numerical approximation, the system~\eqref{eq:sv_change_cons} can be an interesting alternative to existing multilayer models with mass exchange simulating hydrostatic free-surface flows, see~\cite{art_3d} and references therein.

Notice that the multilayer formulation has also to deal with the change of variable that can become singular when time evolves and thus hardly invertible. Therefore a reinterpolation of the variables $(H_\alpha,\tilde u_\alpha)$ has to be done when quantities vary greatly in the domain under consideration. The total mass, the total momentum and the total energy of a column being conserved for smooth solutions, the reinterpolations should also conserve the corresponding quantities. A comparison of performance (computational cost versus accuracy) with classical approximation methods (see~\cite{art_3d} in the multilayer context) would be useful.

We also note that \eqref{eq:sv_change_cons} was also obtained by Bardos-Besse in \cite[§5.0.2]{Bardos-VDB2013} with a muti-kinetic Ansatz in the Vlasov-Dirac-Benney equation.

A similar multilayer formulation was proposed and analyzed in \cite{chesnokov2017} for $\mathbb{P}_1$ approximations.
\end{remark}

\subsection{Characterization of the spectrum in the discrete case}

It is important to study the hyperbolic nature of the multilayer system \eqref{eq:sv_change}, in particular in view of analyzing the stability of a numerical scheme. To this end, the multilayer model \eqref{eq:sv_change} can be rewritten abstractly as the following quasi-linear system.

\begin{proposition}
System~\eqref{eq:sv_change} can be written in a quasi-linear form
\begin{equation}\label{eq:quasi_multilater}
    \frac{\partial \tilde{\boldsymbol{\mathcal{U}}}}{\partial t} + A_{N}(\tilde{\boldsymbol{\mathcal{U}}})\frac{\partial \tilde{\boldsymbol{\mathcal{U}}}}{\partial x}  = \tilde{\boldsymbol{S}},
\end{equation}
with $\tilde{\boldsymbol{\mathcal{U}}} = (\boldsymbol{H}_N,\tilde{\boldsymbol{U}}_N)^T$ and $\tilde{\boldsymbol{S}} = (\mathbf{0}_N,-g \partial_x z_b \mathbf{1}_N)$ being block vectors and $A_{N}(\tilde{\boldsymbol{\mathcal{U}}})$ being the $2\times 2$ block matrix defined by
\[
    A_{N} (\tilde{\boldsymbol{\mathcal{U}}})= \begin{pmatrix}
    \diag(\tilde{\boldsymbol{U}}_N) & \diag(\boldsymbol{H}_N)\\
    g\mathbf{1}_{N}\otimes\boldsymbol{\gamma}_{N} & \diag(\tilde{\boldsymbol{U}}_N)
    \end{pmatrix},
\]
where $\boldsymbol{\gamma}_{N} =(\gamma_{1},\dots,\gamma_{N})^{T}$.
\end{proposition}

\begin{proof}
The result follows easily by explicitly computing the matrix product.
\end{proof}

The following result corresponds to a discrete version of the definition of the point spectrum given in \cref{thm:spectrum} in the case of flat topography.

\begin{proposition} \label{prop:discrete_eigen}
Let $\tilde{\boldsymbol{\mathcal{U}}} = (\boldsymbol{H}_N,\tilde{\boldsymbol{U}}_N)^T$ be a solution of \eqref{eq:quasi_multilater} with $z_b=0$. If $\boldsymbol{H}_N>0$ and $g>0$ then the matrix $A_{N} (\tilde{\boldsymbol{\mathcal{U}}})$ admits at most $2N$ eigenvalues given by
\begin{equation} \label{eq:matrix_spectrum}
    \sigma(A_{N}) = \left\{ c\in\mathbb{C}\setminus\tilde{\boldsymbol{U}}_N : \sum_{i=1}^N \frac{g \gamma_i H_i}{(\tilde u_i-c)^2} = 1
    \right\} \cup \left\{
    c\in\tilde{\boldsymbol{U}}_N\,:\:\card\bigl(\tilde{\boldsymbol{U}}_N^{-1}(c)\bigr)>1\right\},
\end{equation}
where $\tilde{\boldsymbol{U}}_N^{-1}(c)=\bigl\{j\in\{1,\dots,N\} : \tilde{u}_j=c\bigr\}$ is the set of indices for which $\tilde{u}_j=c$.
In particular if all the $\tilde{u}_i$ are distinct for $i\in\{1,\dots,N\}$, then the second part in the union is empty.
\end{proposition}

In the shallow water regime, \emph{i.e.} when $\tilde{u}_i=\tilde{u}_1$ for all $i$, the spectrum reduces to three elements as explained for the continuous case in \cref{rem:spectrum_saint_venant}.
We note that in view of \cref{cor:multi_layer_exact}, \cref{prop:discrete_eigen} can be viewed as a  corollary of \cref{thm:spectrum} by considering the finite-dimensional subspace of $L^2(I)$ of $\mathbb{P}_0$-approximations and modifying the Lebesgue measure appropriately on this space.
However, we provide here a more straightforward proof.

\begin{proof}
In view of the structure of the matrix $A_{N}(\tilde{\boldsymbol{\mathcal{U}}})$, its characteristic polynomial is:
\[
    \det \bigr(A_{N}(\tilde{\boldsymbol{\mathcal{U}}})-c I_{2N}\bigl) = \det \bigr(A_{N}(\tilde{\boldsymbol{\mathcal{U}}}-c(\mathbf{0}_N,\mathbf{1}_N)\bigl)
\]
so by shifting $\tilde{\boldsymbol{U}}_N \mapsto \tilde{\boldsymbol{U}}_N - c \mathbf{1}_N$ we can compute the characteristic polynomial for $c=0$ without loss of generality.
Since $\diag(\tilde{\boldsymbol{U}}_{N})$ and $\diag(\boldsymbol{H}_{N})$
commute, using \cite[equation (16)]{silvester} we obtain
\[
\det (A_{N}\tilde{\boldsymbol{\mathcal{U}}})=\det\bigl(\diag(\tilde{\boldsymbol{U}}_{N})\diag(\tilde{\boldsymbol{U}}_{N})-g\mathbf{1}_N\otimes\boldsymbol{\gamma}_N\diag(\boldsymbol{H}_{N})\bigr)=\det\bigl(\diag(\tilde{\boldsymbol{U}}_{N})^{2}-\mathbf{1}_{N}\otimes\boldsymbol{g}_{N}\bigr),
\]
where $\boldsymbol{g}_{N}=g\boldsymbol{\gamma}_N\odot\boldsymbol{H}_N$ denotes the element-wise or Hadamard product of $\boldsymbol{\gamma}_N$ with $\boldsymbol{H}_N$, \emph{i.e} $g_\alpha=g\gamma_{\alpha}H_{\alpha}$.
Hence $A_{N}(\tilde{\boldsymbol{\mathcal{U}}})$ is invertible if and
only if $\diag(\tilde{\boldsymbol{U}}_{N})^{2}-\mathbf{1}_{N}\otimes\boldsymbol{g}_{N}$
is invertible.
Now one has to distinguish the cases where $\diag(\tilde{\boldsymbol{U}}_{N})$ is invertible or not:
\begin{enumerate}
    \item If $\diag(\tilde{\boldsymbol{U}}_{N})$ is invertible, by using the Sherman--Morrison formula~\cite{SMW}, this is equivalent
    to the fact that $1-\mathbf{1}_{N}\cdot\diag(\tilde{\boldsymbol{U}}_{N})^{-2}\boldsymbol{g}_{N}\neq0$.
    This later condition writes equivalently 
    \[
        1-\mathbf{1}_{N}\cdot\diag(\tilde{\boldsymbol{U}}_{N})^{-2}\boldsymbol{g}_{N}=1-\tilde{\boldsymbol{U}}_{N}^{\odot-2}\cdot\boldsymbol{g}_{N}=1-\sum_{i=1}^{N}\frac{g\gamma_{i}H_{i}}{\tilde{u}_{i}^{2}} \neq 0,
    \]
    where $\tilde{\boldsymbol{U}}_{N}^{\odot-2}=(\tilde{u}_1^{-2},\dots,\tilde{u}_N^{-2})$ denotes the element-wise or Hadamard power.
    Performing the shift back in $\tilde{\boldsymbol{U}}$ leads to the first part in the union given by \eqref{eq:matrix_spectrum}.
    
    \item If $\diag(\tilde{\boldsymbol{U}}_{N})$ is not invertible, then one of its diagonal element is zero, \emph{i.e.} there exists $j\in\{0,\dots,N\}$ such that $\tilde{u}_j=0$.
    Therefore the $j$-th row of the matrix $\diag(\tilde{\boldsymbol{U}}_{N})^{2}-\mathbf{1}_{N}\otimes\boldsymbol{g}_{N}$ is exactly $-\boldsymbol{g}_N^T$ and by subtracting this row to all the others, we obtain that $\diag(\tilde{\boldsymbol{U}}_{N})^{2}-\mathbf{1}_{N}\otimes\boldsymbol{g}_{N}$ is invertible precisely when $\diag(\tilde{\boldsymbol{U}}_{N})^{2}-\mathbf{e}_j\otimes\boldsymbol{g}_N$ is invertible, where $\boldsymbol{e}_j=(\delta_{1j},\dots,\delta_{Nj})$ is the $j$-th basis vector.
    Therefore we proved that $A_{N}(\tilde{\boldsymbol{\mathcal{U}}})$ is invertible if and only if:
    \[
        \det\bigl(\diag(\tilde{\boldsymbol{U}}_{N})^{2}-\mathbf{e}_j\otimes\boldsymbol{g}_N\bigr) = -g \gamma_j H_j \prod_{i\neq j}\tilde{u}_i^2 \neq 0.
    \]
    Since by assumption $H_i>0$, we obtain that $c=0$ is an eigenvalue if and only if $\tilde{u}_i=\tilde{u}_j=0$ for some $i\neq j$, which is precisely the condition $\card\bigl(\tilde{\boldsymbol{U}}_N^{-1}(0)\bigr)>1$.
    Performing the shift back in $\tilde{\boldsymbol{U}}$ leads to the second part in the union given by \eqref{eq:matrix_spectrum}.
\end{enumerate}
\end{proof}

In the discrete case, we can slightly strengthen \cref{prop:spectrum_localization,prop:limit_case}.
\begin{proposition} \label{prop:discrete_localization}
For $\boldsymbol{H}_N>0$, the eigenvalues $\sigma(A_{N})$ of $A_N$ are contained in the union of the following $N+2$ sets:
\begin{align*}
    J_{-} &= \bigl[\tilde{u}_{-}-\sqrt{gh_N},\tilde{u}_{-}\bigr), &
    J_{+} &= \bigl(\tilde{u}_{+},\tilde{u}_{+}+\sqrt{gh_N}\bigr],
\end{align*}
\[
    D_{i}=\left\{ z\in\mathbb{C}\,:\:|z-\tilde{u}_{i}|\leq\sqrt{gh_{N}} \text{ and } \tilde{u}_{-}\leq\Re z\leq\tilde{u}_{+}\right\},
\]
for $1\leq i\leq N$ (illustrated on \cref{fig:eigen_localization} for $N=4$), where
\begin{align*}
    \tilde{u}_{-} &= \inf_\alpha \tilde{u}_\alpha, &
    \tilde{u}_{+} &= \sup_\alpha \tilde{u}_\alpha, &
    h_N = \sum_{i=1}^{N}\gamma_{i}H_{i}.
\end{align*}
Moreover, there exist exactly two eigenvalues $c_{\pm}$ in the intervals $J_{\pm}$:
\[
    \sigma(A_{N})\cap J_{\pm} = \{c_{\pm}\}.
\]
\end{proposition}

\begin{figure}
\centering
\begin{tikzpicture}[scale=1.2]
    \draw[thick,->] (-4.5,0) -- (4.5,0) node[right] {$\Re c$};
    \draw[thick,->] (0,-2.) -- (0,2) node[above] {$\Im c$};
    \draw (-0.08,1.5) node[right] {$+\sqrt{gh_N}$};
    \draw (-0.08,-1.5) node[right] {$-\sqrt{gh_N}$};
    \begin{scope}
        \clip (-2.5,-1.5) rectangle (2.5,1.5);
        \fill[nearly transparent](-2.5,0) circle(1.5);
    	\draw (-2.6,1.5) node[below right] {$D_1$};
    	\fill[nearly transparent](-0.8,0) circle(1.5);
    	\draw (-0.8,1.5) node[below] {$D_2$};
    	\fill[nearly transparent](1.6,0) circle(1.5);
    	\draw (1.6,1.5) node[below left] {$D_3$};
    	\fill[nearly transparent](2.5,0) circle(1.5);
    	\draw (2.6,1.5) node[below left] {$D_4$};
    \end{scope}
    \fill[nearly transparent](-4,-0.05) rectangle (-2.5,0.05);
    \draw (-3.25,0) node[below] {$J_-$};
    \fill[nearly transparent](4,-0.05) rectangle (2.5,0.05);
    \draw (3.25,0) node[below] {$J_+$};
    \draw[fill] (-2.5,0) circle (2pt);
    \draw (-2.5,0) node[below] {$\tilde{u}_1$};
    \draw[fill] (-0.8,0) circle (2pt);
    \draw (-0.8,0) node[below] {$\tilde{u}_2$};
    \draw[fill] (1.6,0) circle (2pt);
    \draw (1.6,0) node[below] {$\tilde{u}_3$};
    \draw[fill] (2.5,0) circle (2pt);
    \draw (2.5,0) node[below] {$\tilde{u}_4$};
\end{tikzpicture}
\caption{Example for $N=4$: the spectrum $\sigma(A_N)$ is included in the intervals $J_{\pm}$ and disks $D_i$ defined in \cref{prop:discrete_localization}. Moreover exactly one eigenvalue $c_{\pm}$ is in each interval $J_{\pm}$.}
\label{fig:eigen_localization}
\end{figure}

\begin{proof}
In view of \cref{prop:discrete_eigen}, the result reduces to the study of the function $F_N$ defined on $\mathbb{C}\setminus\tilde{\boldsymbol{U}}_N$ by:
\begin{equation}\label{eq:F_N_c}
    F_N(c) = \sum_{i=1}^N \frac{g \gamma_i H_i}{(\tilde u_i-c)^2}.
\end{equation}
We note that $F_N(c)$ is a Riemann sum approximation of $F(c)$ given by \eqref{eq:F_c}.
As in the proof of \cref{prop:spectrum_localization}, the first step is to prove:
\[
    \sigma(A_{N})\subset\left\{ z\in\mathbb{C}\,:\:\tilde{u}_{-}\leq\Re z\leq\tilde{u}_{+}\right\} \cup \R.
\]
Since
\[
    \Im F_N(c) = \sum_{i=1}^N \frac{g \gamma_i H_i (\Re c-\tilde{u}_i)\Im c}{|\tilde u_i-c|^2},
\]
if $\Re c>\tilde{u}_{+}$ or $\Re c<\tilde{u}_{-}$, then
$(\Re c-\tilde{u})$ is either strictly positive or strictly negative
on $I$, so the only way to make the sum $\Im F(c)$ zero, is
that $\Im c=0$.

The second step is to prove that the eigenvalues are in the union of the disks $\{D_i\}_{i=1}^{N}$. If $c$ is not in the union of theses disks, then we have $|c-\tilde{u}_i|>\sqrt{g h_n}$ for all $1\leq i\leq N$, and therefore
\[
    |F_N(c)| \leq \sum_{i=1}^N \frac{g \gamma_i H_i}{|\tilde u_i-c|^2} < \sum_{i=1}^N \frac{g \gamma_i H_i}{g h_N} = 1,
\]
and $c$ cannot satisfy $F_N(c)=1$.
We note that this second step is exactly given by the Gershgorin circle theorem.

Finally, the third step is to show the existence of $c_{\pm}\in J_{\pm}$ such that $\sigma(A_{N})\cap J_{\pm}=\{c_{\pm}\}$. The strategy of proof is completely similar to the one in \cref{prop:limit_case} so we only state the main steps. We have:
\begin{align*}
    \lim_{c\to \tilde{u}_{-}}F_N(c)&=\infty, &
    \lim_{c\to \tilde{u}_{+}}F_N(c)&=\infty, &
    F_N(\tilde{u}_{\pm}\pm\sqrt{gh})\le1,
\end{align*}
and since $F$ is strictly decreasing on $J_{+}$ and strictly
increasing on $J_{-}$, there exists exactly one solution
of $F(c)=1$ in $J_{-}$ and exactly one in $J_{+}$.
\end{proof}

In contrast to the continuous case, the previous result is not characterizing all real eigenvalues, as real solutions of summation condition could be in the interval $[\tilde{u}_{-},\tilde{u}_{+}]$. This can be ruled out, under the following assumptions.

\begin{proposition} \label{prop:discret_spectrum}
If $\boldsymbol{H}_N>0$, and either:
\begin{enumerate}
    \item $\tilde{u}_{+}-\tilde{u}_{-} < \sqrt{gh_N}$, or
    \item $\max_i(|\tilde{u}_{i}-\tilde{u}_{i+1}|^2) < 8g \min_{i}(\gamma_{i}H_{i})$
\end{enumerate}
then:
\[
    \sigma(A_{N})\cap \R=\left\{ c_{-}, c_{+} \right\}
    \cup
    \left\{c\in\tilde{\boldsymbol{U}}_N\,:\:\card\bigl(\tilde{\boldsymbol{U}}_N^{-1}(c)\bigr)>1\right\},
\]
and in particular if all the $\tilde{u}_i$ are distinct this reduces to $\sigma(A_{N})\cap \R=\left\{ c_{-}, c_{+} \right\}$.
\end{proposition}

\begin{proof}
We treat each case separately:
\begin{enumerate}
    \item If $(\tilde{u}_{+}-\tilde{u}_{-})^{2} < gh_N$, then we have for all $c\in(\tilde{u}_{-},\tilde{u}_{+})$
    \[
        F_{N}(c)\geq\sum_{i=1}^{N}\frac{g\gamma_{i}H_{i}}{(\tilde{u}_{+}-\tilde{u}_{-})^{2}}\geq\frac{gh_{N}}{(\tilde{u}_{+}-\tilde{u}_{-})^{2}} > 1.
    \]
    \item Let $j\in\{1,N-1\}$, $m_j=\min(\tilde{u}_{j},\tilde{u}_{j+1})$, $M_j=\max(\tilde{u}_{j},\tilde{u}_{j+1})$.
    For $c\in(m_j,M_j)$, we have
    \[
        F_{N}(c)=\sum_{i=1}^{N}\frac{g\gamma_{i}H_{i}}{(c-\tilde{u}_{i})^{2}}\geq g\min_{i}(\gamma_{i}H_{i})\sum_{i=1}^{N}\frac{1}{(c-\tilde{u}_{i})^{2}}\geq\min_{i}(\gamma_{i}H_{i})F_{j}(c),
    \]
    where
    \[
        F_{j}(c)=\frac{1}{(c-\tilde{u}_{j})^{2}}+\frac{1}{(c-\tilde{u}_{j+1})^{2}}.
    \]
    On the interval $(m_j, M_j)$, the function $F_{j}$
    has a minimum at $c^{*}=\frac{\tilde{u}_{j}+\tilde{u}_{j+1}}{2}$,
    so
    \[
        F_{j}(c)\geq F_{j}(c^{*})=\frac{8}{(\tilde{u}_{j+1}-\tilde{u}_{j})^{2}}\geq\frac{8}{\max_{i}|\tilde{u}_{i}-\tilde{u}_{i+1}|^{2}}.
    \]
    Therefore, this proves that for $c\in[m_j,M_j]$, $F(c)>1$.
    Since
    \[
        \bigcup_{j=1}^{N-1}(m_{j},M_{j})\subset(\tilde{u}_{-},\tilde{u}_{+})
    \] we also have that $F(c)>1$ for all $c\in(\tilde{u}_{-},\tilde{u}_{+})$, which concludes the proof
\end{enumerate}

\end{proof}

The previous proposition provides cases where only two of the $2N$ eigenvalues are real, hence all the other have strictly nonzero imaginary part, so the system is in some sense unstable (not strictly hyperbolic).
The following proposition deals with the opposite case where there are exactly $2N$ distinct real eigenvalues.
Unfortunately, the condition is not really interesting for applications as it requires the discretized velocities to be far from each other, but nicely complement the previous proposition.
\begin{proposition} \label{prop:discret_spectrum_real}
If all the $\tilde{u}_{i}$ are distinct and $|\tilde{u}_{i}-\tilde{u}_{j}|>\sqrt{gh_{N}}$, then $\sigma(A_{N})$ has exactly $2N$ distinct real eigenvalues, so the system is strictly hyperbolic.
\end{proposition}
\begin{proof}
Since we only need to study the function $F_{N}(c)$ defined in \eqref{eq:F_N_c} and since all the $\tilde{u}_i$ are distinct, without lost of generality we can label the indices such that
\[
    \tilde{u}_{1}<\tilde{u}_{2}<\dots<\tilde{u}_{N}.
\]
For $1\leq i\leq N-1$, let us define the following mean values
\[
    \tilde{u}_{i+1/2}=\frac{\tilde{u}_{i}+\tilde{u}_{i+1}}{2}.
\]
On the one hand, since
\[
    |\tilde{u}_{j}-\tilde{u}_{i+1/2}|=\frac{1}{2}|\tilde{u}_{j}-\tilde{u}_{i}|+\frac{1}{2}|\tilde{u}_{j}-\tilde{u}_{i+1}|>\sqrt{gh_{N}},
\]
we deduce that
\[
    F_{N}(\tilde{u}_{i+1/2})=\sum_{j=1}^{N}\frac{g\gamma_{i}H_{i}}{(\tilde{u}_{j}-\tilde{u}_{i+1/2})^{2}}<\sum_{j=1}^{N}\frac{g\gamma_{i}H_{i}}{gh_{N}}=1.
\]
On the other hand, we have
\[
    \lim_{c\to\tilde{u}_{i}}F_{N}(c)=\infty.
\]
Therefore, by the mean value theorem, for any $1\leq i\leq N-1$, the equation
$F_{N}(c)=1$ has at least one real solution in the interval $(\tilde{u}_{i},\tilde{u}_{i+1/2})$
and at least another real solution in $(\tilde{u}_{i+1/2},\tilde{u}_{i+1})$.
With the two eigenvalues $c_{\pm}\in J_{\pm}$ provided by
\cref{prop:discrete_localization}, we have found $2N$ distinct real
solutions of $F_{N}(c)=1$.
Therefore the spectrum of $A_{N}$ is purely real:
\[
    \sigma(A_{N})\subset J_{-}\cup(\tilde{u}_{0},\tilde{u}_{1/2})\cup(\tilde{u}_{1/2},\tilde{u}_{1})\cup\dots\cup(\tilde{u}_{N-1/2},\tilde{u}_{N})\cup J_{+}
\]
with exactly one eigenvalue in each interval.
\end{proof}

\subsection{A convergence result of the spectrum to the continuous case}
A natural question is whether the discrete approximation can be linked in some way to the continuous case as $N\to\infty$, for example do we have $\lim_{N\to\infty}\sigma(A_N)=\sigma(A_0)$ in some sense?
Since $A_0$ is not compact in general, it is not trivial to approximate it with finite-rank operators and obtain some convergence of the spectrum.
However, we note that for $g=0$, the spectrum of $A_N$ is $\tilde{\boldsymbol{U}}_N$ and so by a suitable discretization, we have $\lim_{N\to\infty}\sigma(A_N)=\sigma(A_0)=\tilde{u}(\bar{I})$, in the following sense:
\[
    \lim_{N\to\infty} d_H(\sigma(A_N),\sigma(A_0)) = \lim_{N\to\infty} \sup_{\lambda\in I}d(\sigma(A_N),\tilde{u}(\lambda)) = 0,
\]
where $d_H$ denotes the Hausdorff distance.
However, when $g>0$ it seems difficult to prove and maybe wrong that an eigenvalue of $A_N$ can become as close as possible to any point of $\sigma_{e}(A_0)=\tilde{u}(\bar{I})$, but also that solutions of $F_N(c)=1$ converge to solutions of $F(c)=1$, where $F_N(c)$ and $F(c)$ are respectively defined by \eqref{eq:F_c} and \eqref{eq:F_N_c}.

In the continuous case, \cref{prop:u_convexe} provides conditions under which the operator has only a real spectrum, hence is hyperbolic in some generalized sense. We propose here to look at what happens when we use the discretized version of the equations on a profile with the same properties. \Cref{prop:discret_spectrum} ensures that only two eigenvalues are real and the others might have a nonzero imaginary part. In the following proposition, we prove that this imaginary part goes to zero when $N$ goes to infinity assuming that the vertical profile of $\tilde u$ is strictly monotonic in $\lambda$ and $\partial_\lambda (H/\partial_\lambda u)\ne 0$.

\begin{proposition}\label{prop:eigen_convergence}
Let $H\in C^{1}(\bar{I})$ and $\tilde{u}\in C^2(\bar{I})$. For any $N\geq1$, let $\{u_i\}_{1\le i\le N}$ and $\{H_i\}_{1\le i\le N}$ be the $\mathbb{P}_0-$approximation in $\lambda$ of $\tilde{u}$ and $H$ for $\gamma_\alpha=\frac{1}{N}$. If $H>0$ and $\tilde{u}$ is strictly monotonic in $\lambda$ with $\partial_\lambda ( H/ \partial_\lambda \tilde{u})\ne 0$ for all $\lambda\in I$, then for $N$ large enough
\[
    \sup_{c\in\sigma(A_N)} |\Im c| \leq \left(\frac{3gC^3}{N}\right)^{1/4},
\]
where $C>0$ depends only on $\tilde{u}$ and $H$:
\[
    C = \max\bigl(1, \|H\|_{L^\infty}, \|\partial_\lambda H\|_{L^\infty}, \|\tilde{u}\|_{L^\infty}, \|\partial_\lambda \tilde{u}\|_{L^\infty}\bigr).
\]
\end{proposition}

\begin{proof}
In view of \cref{prop:discrete_localization}, it suffices to consider $c\in\mathbb{C}$ such that $|\Im c|\neq0$ and $\tilde{u}_{-}\leq\Re c\leq\tilde{u}_{+}$.

Since $F_N(c)$ is a Riemann sum approximation of $F(c)$, the first step is to prove that:
\begin{equation} \label{eq:riemann_sum}
    |F_{N}(c)-F(c)|\leq\frac{2gC^2(1+|\Im c|)}{|\Im c|^{3}N}.
\end{equation}
We have
\[
    |F_{N}(c)-F(c)|\leq g\sum_{i=1}^{N}\int_{L_{i}}\left|\frac{H_{i}}{(c-\tilde{u}_{i})^{2}}-\frac{H}{(c-\tilde{u})^{2}}\right|\mathrm{d}\lambda.
\]
Using many times the mean value theorem, for any $\lambda\in{L}_{i}$,
we have
\begin{align*}
    \left|\frac{H_{i}}{(c-\tilde{u}_{i})^{2}}-\frac{H(\lambda)}{(c-\tilde{u}(\lambda))^{2}}\right| & \leq\left|\frac{H_{i}-H(\lambda)}{(c-\tilde{u}_{i})^{2}}\right|+H(\lambda)\left|\frac{1}{(c-\tilde{u}_{i})^{2}}-\frac{1}{(c-\tilde{u}(\lambda))^{2}}\right|\\
    & \leq \frac{|H_{i}-H(\lambda)|}{|\Im c|^{2}}+\frac{2H(\lambda)|\tilde{u}_{i}-\tilde{u}(\lambda)|}{|\Im c|^{3}}\\
    & \leq\frac{C}{|\Im c|^{2}N}+\frac{2C^2}{|\Im c|^{3}N},
\end{align*}
since
\begin{align*}
    |H_{i}-H(\lambda)| & \leq CN^{-1}, & |\tilde{u}_{i}-\tilde{u}(\lambda)| & \leq CN^{-1}.
\end{align*}
and therefore \eqref{eq:riemann_sum} is proven.

Since $\tilde{u}$ and $H$ satisfy the hypotheses of \cref{prop:u_convexe}, we know that \eqref{eq:convexe_bound_Fc} holds, so
\[
\Re F(c)\leq\frac{C}{|\Im c|}|\Im F(c)|,
\]
and we deduce that
\begin{align*}
    \Re F_{N}(c) & \leq\Re F(c)+|F_{N}(c)-F(c)|\\
     & \leq\frac{C}{|\Im c|}|\Im F(c)|+|F_{N}(c)-F(c)|\\
     & \leq\frac{C}{|\Im c|}\bigl[|\Im F_{N}(c)|+|F_{N}(c)-F(c)|\bigr]+|F_{N}(c)-F(c)|\\
     & \leq\frac{C}{|\Im c|}|\Im F_{N}(c)|+\left[1+\frac{C}{|\Im c|}\right]|F_{N}(c)-F(c)|.
\end{align*}
In particular if $|\Im c|^4 > 3gC^3/N$ and for $N$ large enough we get
\[
    \Re F_{N}(c) \leq \frac{C}{|\Im c|}|\Im F_{N}(c)|+\left[1+\frac{C}{|\Im c|}\right]\frac{2gC^2(1+|\Im c|)}{|\Im c|^{3}N} < \frac{C}{|\Im c|}|\Im F_{N}(c)| + 1,
\]
which is incompatible with having a solution of $F_N(c)=1$, hence $c\notin\sigma(A_N)$.
\end{proof}

\paragraph{Acknowledgments.}
The authors would like to thank the two anonymous reviewers for their very attentive reading and for their numerous appreciable comments and suggestions, which greatly improved the quality of this article. In particular the Fjørtoft-like condition in \cref{prop:u_convexe} as well as \cref{prop:discret_spectrum_real} was suggested by one of the reviewers.
B. Di Martino acknowledges the SAPHIR project, grant \href{https://anr.fr/Project-ANR-21-CE04-0014}{ANR-21-CE04-0014-03} of the French National Research Agency, for funding part of this work.
J. Guillod acknowledges the support of the Initiative d'Excellence (Idex) of Sorbonne University through the Emergence program.

\appendix

\section{Appendix}\label{sec:appendix}

\begin{definition}
The spectrum of an operator $A:L^{2}(I)\mapsto L^{2}(I)$  is defined as the set $\sigma(A)$ of all $c\in\mathbb{C}$
for which the operator $A-c\mathbf{I}$ is not invertible. The following classification
of the spectrum is used: 
\begin{itemize}
\item The point spectrum $\sigma_{p}(A)$ is defined as the set of all $c\in\mathbb{C}$
for which the operator $A-c\mathbf{I}$ is not injective. 
\item The continuous spectrum $\sigma_{c}(A)$ is defined as the set of
all $c\in\mathbb{C}$ for which the operator $A-c\mathbf{I}$ is injective
and its range is dense in $L^{2}(I)$ but not equal to $L^{2}(I)$. 
\item The residual spectrum $\sigma_{r}(A)$ is defined as the set of all
$c\in\mathbb{C}$ for which the operator $A-c\mathbf{I}$ is injective and
its range is not dense in $L^{2}(I)$.
\item The discrete spectrum $\sigma_{d}(A)$ is defined as the set of all
$c\in\mathbb{C}$ for which the operator $A-c\mathbf{I}$ is not invertible
but Fredholm.
\item The essential spectrum $\sigma_{e}(A)$ is defined as the set of all
$c\in\mathbb{C}$ for which the operator $A-c\mathbf{I}$ is not Fredholm.
\end{itemize}
Therefore: 
\[
\sigma(A)=\sigma_{p}(A)\cup\sigma_{c}(A)\cup\sigma_{r}(A)=\sigma_{d}(A)\cup\sigma_{e}(A).
\]
\end{definition}

\begin{definition}
For $\tilde{u}\in L^{\infty}(I)$, the essential range of $\tilde{u}$,
denoted by $\tilde{u}[I]$ is defined as the set of all points $c\in\mathbb{R}$,
such that for all $\varepsilon>0$, the measure of $\{\lambda\in I\,:\:|\tilde{u}(\lambda)-c)|<\varepsilon\}$
is non-zero. Equivalently, $c\in\tilde{u}[I]$ if and only if $(\tilde{u}-c)^{-1}\notin L^{\infty}(I)$.
\end{definition}

\bibliographystyle{merlin-doi}
\phantomsection\addcontentsline{toc}{section}{\refname}\bibliography{main}
\end{document}